\documentclass[10pt]{smfart}

\RequirePackage[T1]{fontenc}
\RequirePackage{amsfonts,latexsym,amssymb}
\RequirePackage[french]{babel}

\RequirePackage{mathrsfs}
\let\mathcal\mathscr

\usepackage[matrix,arrow]{xy}
\usepackage{url}
\usepackage{accents}
\usepackage{enumitem}
\usepackage{tikz}
\usepackage{fge}
\newcommand{\moins}{\mathbin{\fgebackslash}}

\theoremstyle{plain}
\numberwithin{equation}{section}
\newtheorem{prop}[equation]{\propname}
\newtheorem{theo}[equation]{\theoname}
\newtheorem{conj}[equation]{\conjname}
\newtheorem{coro}[equation]{\coroname}        

\newtheorem{lemm}[equation]{\lemmname}
\theoremstyle{definition}
\theoremstyle{remark}
\newtheorem{hypothesis}[equation]{Hypoth\`ese}

\newtheorem{rema}[equation]{\remaname}
\newtheorem{exem}[equation]{\exemname}

\newtheorem*{acknowledgments}{Remerciements}

\newcommand{\proeet}{\operatorname{pro\acute{e}t} }

\let\cal\mathcal
\let\goth\mathfrak

\def\Q{{\bf Q}} \def\Z{{\bf Z}}
\def\C{{\bf C}}
\def\N{{\bf N}}
\def\R{{\bf R}}
\def\O{{\cal O}}

\def\dual{{\boldsymbol *}}
\def\bmu{{\boldsymbol\mu}}

\def\rg{{\rm R}\Gamma}

\def\Qbar{\overline{\bf Q}}
\def\epsilon{\varepsilon}

\def\ainf{{\bf A}_{{\rm inf}}}

\def\bst{{\bf B}_{{\rm st}}}
\def\bcris{{\bf B}_{{\rm cris}}}

\def\bdr{{\bf B}_{{\rm dR}}}
\def\bpdr{{\bf B}_{{\rm pdR}}}
\def\bff{{\bf B}_{\rm FF}}
\def\bpff{{\bf B}_{\rm pFF}}
\def\Bpff#1{{\bf B}_{{\rm pFF},#1}}

\newcommand{\shprod}{\mathbin{\amalg}}

 \def\A{{\bf A}} \def\B{{\bf B}}

\title{Une conjecture $C_{\rm st}$ pour la cohomologie \`a support compact}
\author{Pierre Colmez}
\address{CNRS, IMJ-PRG, Sorbonne Universit\'e, 4 place Jussieu, 75005 Paris, France}
\email{pierre.colmez@imj-prg.fr}
\author{Sally Gilles}
\address{LMB, Universit\'e Marie et Louis Pasteur, 16 route de Gray, 25030 Besan\c{c}on CEDEX, France}
\email{sally.gilles@umlp.fr}
\author{Wies{\l}awa Nizio{\l}}
\address{CNRS, IMJ-PRG, Sorbonne Universit\'e, 4 place Jussieu, 75005 Paris, France}
\email{wieslawa.niziol@imj-prg.fr}
\thanks{P.\,C. et W.\,N. sont partiellement financ\'es par la Simons Collaboration on Perfection in Algebra, Geometry, and Topology.}

\begin{abstract}
Soit $\B$ l'anneau des fonctions analytiques sur la courbe de Fargues-Fontaine $Y_{\rm FF}$.
Nous montrons qu'en adjoignant suffisamment de p\'eriodes it\'er\'ees, on tue sa
cohomologie galoisienne en degr\'es~$\geq 1$. Pour $\bdr^+$, le r\'esultat analogue
s'obtient en adjoignant simplement $\log t$ et rel\`eve du folklore.
Cela permet de formuler des conjectures de type $C_{\rm dR}$ et $C_{\rm st}$
pour la cohomologie \`a support compact des vari\'et\'es analytiques $p$-adiques.
\end{abstract}
\begin{altabstract}
Let $\B$ be the ring of analytic functions on the Fargues-Fontaine curve $Y_{\rm FF}$.
We shaw that, by adding sufficiently iterated periods, one can kill
its Galois cohomology in degrees~$\geq 1$. For $\bdr^+$, the corresponding result
is obtained by adjoining just $\log t$ and is folklore.
This allows to formulate $C_{\rm dR}$ and  $C_{\rm st}$ conjectures
for the compactly supported cohomology of $p$-adic analytic varieties.
\end{altabstract}
\begin{document}

\maketitle
\tableofcontents

\section*{Introduction}
Les conjectures $C_{\rm dR}$ et $C_{\rm st}$ \cite[conj.\,1.6]{CN5}
fournissent des recettes pour extraire, 
de leur cohomologie pro\'etale $p$-adique, les cohomologies de
de Rham et de Hyodo-Kato des vari\'et\'es analytiques $p$-adiques: 
si ${\rm Hom}_{G_L}$ d\'esigne les
morphismes $G_L$-\'equivariants
 continus (pour les topologies naturelles sur les
deux membres), on a la conjecture suivante.
\begin{conj}\label{cst1}
Si $Y$ est une vari\'et\'e analytique partiellement propre, d\'efinie
sur une extension finie $K$ de $\Q_p$,
on a des isomorphismes fonctoriels {\rm(de $K$-modules filtr\'es et de 
$(\varphi,N,G_K)$-modules sur $\Q_p^{\rm nr}$, respectivement)}
\begin{align*}
&(C_{\rm dR})&{\rm Hom}_{G_K}(H^n_{\proeet}(Y_{\C_p},\Q_p),\bdr)&\simeq H^n_{\rm dR}(Y)^\dual\\
&(C_{\rm st})&\varinjlim\nolimits_{[L:K]<\infty}
{\rm Hom}_{G_L}(H^n_{\proeet}(Y_{\C_p},\Q_p),\bst)&\simeq H^n_{\rm HK}(Y_{\C_p})^\dual
\end{align*}
\end{conj}
Les $H^n_{\proeet}(Y_{\C_p},\Q_p)$
sont en g\'en\'eral de dimension infinie, m\^eme si les $H^n_{\rm dR}(Y)$ sont de dimension finie 
(c'est par exemple
le cas pour le $H^1$ de l'analytifi\'e d'une courbe alg\'ebrique non propre, cf.~\cite[th.\,0.7]{CDN2}
ou~\cite[th.\,1.8]{CDN3}), et la formulation usuelle (covariante) des conjectures
$C_{\rm dR}$ et $C_{\rm st}$ 
(pour la cohomologie \'etale des vari\'et\'es alg\'ebriques) tombe en d\'efaut (elle donnerait des groupes beaucoup
trop gros), 
ce qui explique la formulation (contravariante) ci-dessus. 

Pour la cohomologie \`a support
compact, on a le m\^eme probl\`eme de finitude, mais 
prendre un point de vue contravariant ne suffit pas (rem.\,\ref{cst3}):
au lieu de ${\rm Hom}$, on doit prendre un ${\rm Hom}$ d\'eriv\'e (sinon on obtient des
groupes trop petits),
et aussi tuer la cohomologie galoisienne des anneaux $\bdr$
et $\bst$ en degr\'es~$\geq 1$ (sinon on se retrouve avec des termes parasites).

Dans cette note, on explique comment tuer\footnote{Plus pr\'ecis\'ement,
on veut une alg\`ebre contenant l'anneau consid\'er\'e, ayant m\^eme $H^0$, mais tel
que $H^i=0$ pour $i\geq 1$.}, en degr\'es~$\geq 1$,
la cohomologie galoisienne des anneaux
de p\'eriodes $\bdr^+,\bdr$,
en rajoutant un logarithme du $2i\pi$ $p$-adique
de Fontaine (th.\,\ref{gopal4}),
et celle (des avatars $\B[\log\tilde p], \B[\log\tilde p,\frac{1}{t}]$) de $\bst^+,\bst$ 
en se fatiguant un peu plus (th.\,\ref{tue103}).
Le cas de $\bdr^+$ rel\`eve du folklore mais qu'il existe un objet raisonnable tuant la
cohomologie de
$\B[\log\tilde p]$ est plus surprenant\footnote{Dans une version ant\'erieure de cet article,
nous avions affirm\'e que, quitte \`a remplacer $\bst^+$ par $\B[\log\tilde p]$, il suffisait
de rajouter $\log t$, mais comme nous l'expliquons dans le chap.\,\ref{sect6},
il n'est pas possible de s'en tirer en rajoutant un nombre fini d'\'el\'ements (le lemme~\ref{sushi12}
montre que rajouter un \'el\'ement pour tuer une classe de cohomologie introduit de nouvelles
classes non triviales).}.

Ceci nous permet de formuler (cf.~conj.\,\ref{cst2.1} et~\ref{cst2.2}) 
un analogue de la conjecture~\ref{cst1} 
pour la cohomologie \`a support compact des vari\'et\'es analytiques partiellement propres.
\begin{acknowledgments} Nous remercions Hui Gao pour des \'echanges
int\'eressants concernant cet article.

Lors de la pr\'eparation de ce travail, ChatGPT (OpenAI) a \'et\'e utilis\'e
pour explorer et discuter certains arguments math\'ematiques, notamment
pour d\'etecter une erreur dans une version ant\'erieure et d\'evelopper
la pr\'esentation par les alg\`ebres de m\'elange. 
\end{acknowledgments}

\section{La conjecture $C_{\rm dR}$ pour la cohomologie \`a support compact}\label{sect0}
\Subsection{Cohomologie galoisienne de $\bdr$}\label{sect1}
\Subsubsection{Cohomologie galoisienne de $C$}\label{sect2}
Soient $\kappa$ un corps parfait d\'enombrable, de caract\'eristique $p$, 
$F:=W(\kappa)[\frac{1}{p}]$
et $K$ une extension finie totalement ramifi\'ee de $F$ 
(et donc le corps r\'esidel $\kappa_K$ de $K$
est~$\kappa$).  On note $G_K$ le groupe de Galois absolu de $K$ 
et $C$ le compl\'et\'e d'une cl\^oture
alg\'ebrique de $K$.  On doit \`a Tate~\cite{Ta67}
 le calcul de la cohomologie galoisienne de $C$.
Si $j\in\Z$, on note $C(j)$ le tordu de $C$ par $\chi_{\rm cycl}^j$, o\`u $\chi_{\rm cycl}:G_K\to\Z_p^\dual$
est le caract\`ere cyclotomique.
\begin{prop} {\rm (Tate)}\label{gopal1}
On a
\begin{align*}
H^0(G_K,C(j))=\begin{cases} K&{\text{si $j=0,$}}\\ 0&{\text{si $j\neq 0;$}}\end{cases}
\quad&
H^1(G_K,C(j))\cong\begin{cases} K\cdot \log\chi_{\rm cycl}&{\text{si $j=0,$}}\\ 0&{\text{si $j\neq 0;$}}\end{cases}\\
H^i(G_K,C(j))=0,\ &{\text{pour tout $j$, si $i\geq 2$.}}
\end{align*}
\end{prop}

\Subsubsection{\'Elimination de la cohomologie de $C$ en degr\'es~$\geq 1$}\label{sect3}
On d\'efinit $\log t$ comme \'etant transcendant sur $C$, et on munit $C[\log t]$ de l'action
de $G_K$ donn\'ee par $\sigma(\log t)=\log t+\log\chi_{\rm cycl}(\sigma)$ (comme dans~\cite{Csen}).
\begin{prop}\label{gopal2}
On a
$$H^0(G_K,C[\log t])=K,\quad
H^i(G_K,C[\log t](j))=0 {\text{ si $i\geq 1$ ou si $j\neq 0$.}}$$
\end{prop}
\begin{proof}
Notons $C[\log t]^{\leq k}$ le $C$-module des polyn\^omes de degr\'e~$\leq k$ en $\log t$.
Alors $C[\log t]^{\leq k}$ est stable par $G_K$ et, si $j\in\Z$, on a une suite exacte de $G_K$-modules
$$0\to C[\log t]^{\leq k}(j)\to C[\log t]^{\leq k+1}(j)\to C(j)\to 0$$

$\bullet$
Si $j\neq 0$, on en d\'eduit, par une r\'ecurrence imm\'ediate utilisant
la prop.\,\ref{gopal1}, que pour tout $i$ et tout~$k$, on a
$H^i(G_K,C[\log t]^{\leq k}(j))=0$. En passant \`a la limite
inductive, cela prouve que $H^i(G_K,C[\log t](j))=0$, pour tout $i$.

$\bullet$
Si $j=0$, on 
prouve, par r\'ecurrence sur $k$, que $H^i(G_K,C[\log t]^{\leq k})=0$ si $i\geq 2$,
et donc que $H^i(G_K,C[\log t])=0$ si $i\geq 2$. La suite exacte longue de cohomologie
associ\'ee \`a la suite exacte courte ci-dessus et la nullit\'e de $H^2(G_K,C[\log t]^{\leq k})$
fournissent une suite exacte
$$\xymatrix@R=4mm@C=4mm{
0\ar[r] & H^0(G_K,C[\log t]^{\leq k})\ar[r] & H^0(G_K,C[\log t]^{\leq k+1})\ar[r] &
H^0(G_K,C)\ar[dll]\\
 & H^1(G_K,C[\log t]^{\leq k})\ar[r] & H^1(G_K,C[\log t]^{\leq k+1})\ar[r] &
H^1(G_K,C)\ar[r] & 0}$$
Montrons, par r\'ecurrence sur $k$, que:

$\bullet$ $H^0(G_K,C[\log t]^{\leq k})=K$,

$\bullet$ $H^1(G_K,C[\log t]^{\leq k})\cong K$, 

$\bullet$
$H^0(G_K,C)\to H^1(G_K,C[\log t]^{\leq k})$ est un isomorphisme,

$\bullet$ $H^1(G_K,C[\log t]^{\leq k})\to H^1(G_K,C[\log t]^{\leq k+1})$
est l'application nulle.

Tous les modules ci-dessus sont des $K$-espaces vectoriels et 
l'hypoth\`ese de r\'ecurrence alli\'ee \`a la prop.\,\ref{gopal1} montre qu'ils sont
tous de dimension $1$ sauf peut-\^etre $H^i(G_K,C[\log t]^{\leq k+1})$, pour $i=0,1$.
Si $H^0(G_K,C[\log t]^{\leq k+1})\neq K$, comme $H^0(G_K,C)=K$, il existe $a_{k},\cdots, a_0\in C$ tels que
$(\log t)^{k+1}+a_{k}(\log t)^{k}+\cdots+a_0$ soit fixe par $G_K$. En appliquant $\sigma-1$
et en regardant modulo $C[\log t]^{\leq k-1}$, cela donne 
$(\sigma-1)a_{k}=-(k+1)\log\chi_{\rm cycl}(\sigma)$, ce qui contredit le fait
que la classe du $1$-cocycle $\sigma\mapsto \log\chi_{\rm cycl}(\sigma)$ dans $H^1(G_K,C)$ est 
non nulle (prop.\,\ref{gopal1}).

On en d\'eduit que $H^1(G_K,C[\log t]^{\leq k+1})$ est de dimension~$1$ et tous
les autres \'enonc\'es r\'esultent de l'exactitude de la suite ci-dessus.
Par passage \`a la limite, on en d\'eduit que
$H^0(G_K,C[\log t])=K$ et
$H^1(G_K,C[\log t])=0$. 
\end{proof}

\Subsubsection{\'Elimination de la cohomologie de $\bdr^+$ en degr\'es~$\geq 1$}\label{sect4}
Comme $t^j\bdr^+/t^{j+1}\bdr^+=C(j)$, il n'est pas tr\`es difficile de d\'eduire de la prop.\,\ref{gopal1}
le r\'esultat suivant. 
\begin{prop}\label{gopal3}
Si $\Lambda=\bdr^+,\bdr$, 
$$H^0(G_K,\Lambda)=K,\quad H^1(G_K,\Lambda)=K\cdot \log\chi_{\rm cycl},
\quad H^i(G_K,\Lambda)=0{\text{ si $i\geq 2$}}.$$
\end{prop}
La preuve de la prop.\,\ref{gopal2} peut alors s'adapter pour d\'emontrer le r\'esultat suivant,
dans lequel on a pos\'e (comme dans~\cite{Fpdr}):
$$\bpdr^+:=\bdr^+[\log t],\quad \bpdr:=\bdr[\log t]$$

\begin{theo} \label{gopal4}
Si $\Lambda=\bpdr^+,\bpdr$, 
$$H^0(G_K,\Lambda)=K,
\quad H^i(G_K,\Lambda)=0{\text{ si $i\geq 1$}}.$$
\end{theo}

\Subsection{La conjecture $C_{\rm dR}$}\label{sect5}
On munit $\bpdr$ de la filtration d\'efinie par ${\rm Fil}^i\bpdr:=t^i\bpdr^+$.
\begin{conj}\label{cst2.1}
Si $Y$ est une vari\'et\'e analytique lisse, de dimension~$d$, partiellement propre,
 d\'efinie sur $K$,
on a un isomorphisme naturel de $K$-modules filtr\'es:
\begin{align*}
&(C_{\rm dR})&
\rg_{\rm dR}(Y)&\simeq {\rm RHom}_{G_K}(\rg_{{\proeet},c}(Y_{C},\Q_p(d))[2d],\bpdr)
\end{align*}
En particulier, si $Y$ est Stein, alors
\begin{align*}
& &
\Omega^i(Y)[-i]&\simeq {\rm RHom}_{G_K}(\rg_{{\proeet},c}(Y_{C},\Q_p(d))[2d],\C_p(i)[\log t])
\end{align*}
\end{conj}

\begin{rema}\label{cst3}
(i) On peut mettre la conj.\,\ref{cst1} sous une forme plus proche de celle ci-dessus
en utilisant la dualit\'e de Poincar\'e pour la cohomologie de de Rham.

(ii)  Pour formuler une conjecture comme ci-dessus, la premi\`ere id\'ee qui vient
\`a l'esprit
serait de consid\'erer
${\rm Hom}_{G_K}(H^i_{{\proeet},c}(Y_{C},\Q_p),\bdr)$ (apr\`es tout, \c{c}a marche si $Y$ est propre!).
Le probl\`eme est que, si $Y$ est une courbe affine, un petit calcul montre que
${\rm Hom}_{G_K}(H^1_{{\proeet},c}(Y_{C},\Q_p),\bdr)$ ne contient qu'une petite partie
de $H^1_{\rm dR}(Y)$ (celle correspondant au sous-espace de pente~$1$  
de $H^1_{\rm HK}(Y_{C})$, cf.~rem.\,\ref{cst17} ci-dessous).

(iii) La seconde id\'ee serait de consid\'erer un ${\rm RHom}_{G_K}$ plut\^ot qu'un ${\rm Hom}_{G_K}$.
Le cas
d'une vari\'et\'e propre montre que ce n'est pas encore \c{c}a car, si
$V$ est une repr\'esentation de de Rham, ${\rm Ext}_{G_K}^1(V,\bdr)\cong {\rm Ext}_{G_K}^0(V,\bdr)$
(en interpr\'etant ces groupes comme $H^i(G_K,\bdr\otimes V^\dual)$ et en utilisant
le fait que $\bdr\otimes V^\dual\cong\bdr^{\oplus d}$ puisque $V$ est de Rham,
l'isomorphisme est le cup-produit avec $\log\chi_{\rm cycl}$).  Il s'ensuit
que l'on va retrouver deux fois les groupes de cohomologie de de Rham au lieu
d'une... Il faut donc tuer le ${\rm Ext}_{G_K}^1$ pour les repr\'esentations de de Rham
(et donc rajouter $\log t$ pour trivialiser $\log\chi_{\rm cycl}$).
Cela conduit \`a la conjecture $C_{\rm dR}$ ci-dessus, et on s'est d\'ebrouill\'e pour que 
{\it cette
conjecture soit un th\'eor\`eme si $Y$ est propre}.

(iv) Pour formuler $C_{\rm st}$ l'id\'ee naturelle serait de remplacer $\bdr$ par $\bst$
comme pour la conj.\,\ref{cst1}, mais la cohomologie de $\bst[\log t]$ n'est pas $0$
en degr\'e~$1$ (elle est m\^eme de dimension infinie), et donc on retombe sur le probl\`eme mentionn\'e au point (iii).
D\'efinir un anneau $\bpff$
permettant de formuler la conjecture demande nettement plus de travail
(cf.~\S\,\ref{sect11}) .

(v) Implicitement, ${\rm RHom}_{G_K}$ est calcul\'e dans une cat\'egorie ad\'equate
de $\Q_p$-espaces vectoriels topologiques munis d'une action continue de $G_L$.
Le choix \'evident, \`a savoir la cat\'egorie des $\Q_p$-modules condens\'es munis d'une action continue de $G_K$,
ne marche pas car $\dim_K{\rm Ext}^1(\C_p,\C_p)=\infty$ dans cette cat\'egorie,
alors que l'on voudrait que ${\rm Ext}^1(\C_p,\C_p)=K$ comme dans la th\'eorie des
presque $\C_p$-repr\'esentations de Fontaine~\cite{FCp}. 
Dans notre travail en cours~\cite{CGLN}, nous utilisons une version $G_K$-\'equivariante
de la cat\'egorie des TVS que nous avons utilis\'ee dans notre preuve
des th\'eor\`emes de dualit\'e~\cite{CGN2}.
\end{rema}

\begin{rema}\label{cst17}
Les calculs ci-dessous sont ceux qui nous ont donn\'e confiance en la conjecture; ces calculs sont purement
heuristiques et ne constituent pas une preuve dans le cas consid\'er\'e\footnote{Une vraie preuve
demanderait d'utiliser des techniques nettement plus sophistiqu\'ees comme
dans~\cite{CGN}, par exemple,
o\`u l'on examine ce qui se passe si on remplace $\bpdr$ par $\Q_p$.}.

Soit $Y$ l'analytifi\'e d'une courbe alg\'ebrique lisse, 
non propre, d\'efinie sur une extension finie $K$ de $\Q_p$;
les $H^i_{\rm dR}(Y)$ sont des $K$-modules de rang fini, nuls si $i\geq 2$.
Les th\'eor\`emes de comparaison~\footnote{Pour les courbes, on peut utiliser~\cite[th.\,0.7]{CDN2} 
ou~\cite[th.\,1.8]{CDN3}.}  et la suite exacte longue reliant
la cohomologie \`a support compact, la cohomologie usuelle et celle du bord\footnote{Maintenant que la
th\'eorie a \'et\'e d\'evelopp\'ee,
on pourrait aussi utiliser les th\'eor\`emes de comparaison pour la cohomologie
\`a support compact~\cite[\S\,7.2]{AGN}, qui donnent le r\'esultat sous une forme un peu
diff\'erente.}, nous donnent les isomorphismes
et suite exacte suivants, dans lesquels on utilise l'abr\'eviation
$H^i_c:=H^i_{{\proeet},c}(Y_{C},\Q_p(1))$,
\begin{align*}
H^i_c=0,{\text{ si $i\neq 1,2$}},\quad H_c^1 \simeq V,
\quad 0\to V_1 \to H_c^2\to V_2\to 0
\end{align*}
o\`u: 

$\bullet$ $V:=(H^1_{{\rm HK},c}(Y_{C})\otimes t\bst^+)^{N=0,\varphi=p}$ est de dimension finie sur $\Q_p$, et est de Rham
en tant que repr\'esentation de $G_K$.

$\bullet$ $V_1$ est le quotient de $H^1_{{\rm dR},c}(Y_{C})$ par l'image de
$V'_1:=(H^1_{{\rm HK},c}(Y_{C})\otimes \bst^+)^{N=0,\varphi=p}$ (cette image est $V'_1/V$);
c'est une presque $C$-repr\'esentation au sens de Fontaine.

$\bullet$ $V_2$ est une extension de $\Q_p$ par le ${C}$-dual topologique $(\O(Y_{C})/C)^\dual$ de
$\O(Y_{C})/{C}$ (un ${C}$-module topologique de rang infini).

Posons ${\rm Ext}^i:={\rm Ext}^i_{G_K}$ et
$${\rm RHom}^i:=H^i({\rm RHom}_{G_K}(\rg_{{\proeet},c}(Y_{C},\Q_p(1))[2],\bpdr))$$
Comme ${\rm Ext}^2(-,\bdr)$ devrait \^etre $0$, la suite spectrale convergeant vers ${\rm RHom}^i$ fournit
des suites exactes, pour $i\in \Z$,
\begin{equation}\label{rhom}
0\to {\rm Ext}^1(H^{3-i}_c)\to {\rm RHom}^i\to {\rm Ext}^0(H^{2-i}_c)\to 0
\end{equation}
Comme ${\rm Ext}^0({C},\bpdr)=0$ et ${\rm Ext}^1({C},\bpdr)=0$, 
on a 
\begin{align*}
{\rm Ext}^0(V_1,\bpdr)=0,\quad {\rm Ext}^1(V_1,\bpdr)={\rm Ext}^0(V'_1/V,\bpdr)\\ 
{\rm Ext}^0((\O(Y_{C})/{C})^\dual,\bpdr)=0,\quad {\rm Ext}^1((\O(Y_{C})/{C})^\dual,\bpdr)=0
\end{align*}
Les seuls termes non nuls\footnote{Le calcul des ${\rm Ext}^0$ utilise~\cite[th.\,5.8]{CN5}.}
 dans les suites exactes~(\ref{rhom}) sont donc\footnote{
On a des isomorphismes fonctoriels (de Hyodo-Kato) $H^i_{\rm dR}(Y)\simeq (\Qbar_p\otimes H^i_{\rm HK}(Y_{C}))^{G_K}$
et $H^i_{{\rm dR},c}(Y)\simeq (\Qbar_p\otimes H^i_{{\rm HK},c}(Y_{C}))^{G_K}$. Les pentes de Frobenius sur
$H^1_{\rm HK}(Y_{C})$ et $H^1_{{\rm HK},c}(Y_{C})$ sont dans l'intervalle 
$[0,1]$; si $?\in\{\ ,c\}$, on note $H^1_{{\rm HK},?}(Y_{C})^{(r)}$
le sous-espace de pente $r$ 
et on pose $H^1_{{\rm dR},?}(Y):=(\Qbar_p\otimes H^1_{{\rm HK},?}(Y_{C})^{(r)})^{G_K}$;
la dualit\'e entre $H^1_{\rm dR}(Y)$ et $H^1_{{\rm dR},c}(Y)$ met en dualit\'e
$H^1_{\rm dR}(Y)^{(r)}$ et $H^1_{{\rm dR},c}(Y)^{(1-r)}$. On a aussi 
$V=t(H^1_{{\rm HK},c}(Y_{C})^{(0)}\otimes\bcris^+)^{\varphi=1}$.}
\begin{align*}
{\rm Ext}^0(H_c^1,\bpdr)=
(H^1_{{\rm dR},c}(Y)^{(0)})^\dual &=H^1_{\rm dR}(Y)^{(1)}
\subset H^1_{\rm dR}(Y)\\
{\rm Ext}^0(H_c^2,\bpdr)=K=H^0_{\rm dR}(Y),\quad
{\rm Ext}^1&(H_c^2,\bpdr)={\rm Ker}\big(H^1_{\rm dR}(Y)\to (H^1_{{\rm dR},c}(Y)^{(0)})^\dual\big)
\end{align*}
Injecter ce r\'esultat dans (\ref{rhom}) fournit des suites exactes donnant corps \`a la conjecture~\ref{cst2.1}
qui affirme que ${\rm RHom}^i\simeq H^i_{\rm dR}(Y)$, pour tout~$i$.
(Si on avait utilis\'e $\bdr$ au lieu de $\bpdr$, il y aurait des termes parasites provenant de
${\rm Ext}^1(V,\bdr)$ et ${\rm Ext}^1(\Q_p,\bdr)$ contribuant respectivement \`a
${\rm Ext}^1(H^1,\bdr)$ et ${\rm Ext}^1(H^2,\bdr)$.)
\end{rema}

\section{Une construction universelle pour tuer $H^1(G,B)$}\label{sect6}
\Subsection{Comment tuer la cohomologie des anneaux de p\'eriodes?}
Pour formuler une conjecture $C_{\rm st}$ pour la cohomologie \`a support compact,
on a besoin de tuer la cohomologie de $\bcris$ (ou d'une variante) en degr\'es~$\geq1$.
La recette que nous avons utilis\'ee pour tuer celle de $\bdr$ ne marche pas telle quelle, comme
nous l'expliquons ci-dessous.  Pour pouvoir traiter les variantes \'eventuelles de $\bcris$,
consid\'erons un cadre un peu plus g\'en\'eral.

Soit donc $B$ une $\Q$-alg\`ebre topologique 
munie d'une action continue d'un groupe profini $G$, avec $F:=H^0(G,B)$
un corps valu\'e complet, de caract\'eristique~$0$, et $H^i(G,B)=0$ si $i\geq 2$.

Si $\dim_F H^1(G,B)=1$,
la recette pour tuer la cohomologie de $\bdr^+$ en degr\'es~$\geq 1$ 
(i.e.~rajouter un \'el\'ement transcendant qui tue la classe non triviale dans $H^1(G,B)$)
marche
pour tuer $H^i(G,B)$ pour $i\geq 1$.
Par contre, si $\dim_F H^1(G,B)\geq 2$, le lemme~\ref{sushi12} ci-dessous montre
que l'on ne peut pas tuer $H^1(G,B)$ en rajoutant un nombre fini d'\'el\'ements
(ceci s'applique \`a $\bcris^+$ -- et \`a ses variantes -- car son $H^1$ contient 
au moins deux classes ind\'ependantes (\`a savoir
celles associ\'ees \`a $\log\tilde p$ et $\log t$)).
Il existe n\'eanmoins une recette universelle pour tuer cette cohomologie:
nous en donnons deux pr\'esentations dans la suite cet article:

\vskip1mm
$\bullet$  L'id\'ee naturelle est de prendre une base $c_1,\dots, c_r$ de $H^1(G,B)$ sur $F$,
et de rajouter des variables $Y_i$ v\'erifiant $\sigma(Y_i)=Y_i+c_{i,\sigma}$
pour $1\leq i\leq r$ (o\`u $\sigma\mapsto c_{i,\sigma}$ est un $1$-cocycle dans la classe de $c_i$).
Si $B_1=B[Y_1,\dots,Y_r]$, les $c_i$ meurent dans $H^1(G,B_1)$, mais $H^1(G,B_1)\neq 0$
(et m\^eme est de dimension infinie si $r\geq 2$). On peut alors recommencer et construire
une $B_1$-alg\`ebre polynomiale $B_2$ telle que $H^1(G,B_1)\to H^1(G,B_2)$ soit identiquement~$0$,
etc. \`A la limite, on obtient une $B$-alg\`ebre $B_\infty$ ayant les propri\'et\'es requises
(cf.~rem.\,\ref{beurk}). L'inconv\'enient de cette m\'ethode est que les alg\`ebres $B_n$
qui apparaissent sont de type infini sur $B$, ce qui pose des probl\`emes pour les applications.
Dans ce chapitre, nous pr\'esentons une variante moins brutale de ce processus, pour laquelle
les $B_n$ sont toutes de type fini (mais la limite $B_\infty$ reste de type infini).

\vskip1mm
$\bullet$ Dans l'appendice, nous donnons une pr\'esentation plus conceptuelle
reposant sur la notion d'alg\`ebre de m\'elange. Si $W=H^1(G,B)$,
le produit de m\'elange munit ${\rm Sh}_F(W):=\oplus_{n\geq 0} W^{\otimes n}$
d'une structure d'alg\`ebre commutative (qui peut se raffiner pour
obtenir une alg\`ebre de Hopf, ce qui fait que ${\rm Sh}_F(W)$ est naturellement
l'alg\`ebre des fonctions sur un groupe alg\'ebrique pro-unipotent ${\cal U}$). 
Si $c:W\to Z^1(G,B)$ est un scindage de l'application naturelle, on peut
relever $c$ en un $1$-cocycle $C$ \`a valeurs dans ${\cal U}(B)$, ce qui permet
de tordre l'action naturelle de $G$ sur $B\otimes_F{\rm Sh}_F(W)$,
pour obtenir une $B$-alg\`ebre $P_C$ v\'erifiant $H^1(G,P_C)=0$.

\Subsection{Un exemple}
On suppose que $\dim_F H^1(G,B)\geq 2$.
On peut donc trouver deux $1$-cocycles $\sigma\mapsto c_\sigma$ et $\sigma\mapsto c'_\sigma$
dont les classes $c,c'$ dans $H^1(G,B)$ sont ind\'ependantes sur $F$.
Notons $B'$ la $B$-alg\`ebre $B[U]$ que l'on munit d'une action de $G$ prolongeant
celle sur $B$ via la formule $\sigma(U)=U+c_\sigma$, si $\sigma\in G$.
Le $2$-cocycle
$$(\sigma,\tau)\mapsto
(d(Uc'))_{\sigma,\tau}
=\sigma(Uc'_\tau)-Uc'_{\sigma\tau}+Uc'_\sigma
=(U+c_\sigma)\sigma(c'_\tau)
  -Uc'_{\sigma\tau}+Uc'_\sigma=c_\sigma\,\sigma(c'_\tau)$$
est \`a valeurs dans $B$. Comme $H^2(G,B)=0$, il existe
une $1$-cocha\^{\i}ne $\sigma\mapsto b_\sigma$ telle que $\sigma\mapsto Uc'_\sigma-b_\sigma$
soit un $1$-cocycle.
\begin{lemm}\label{sushi12}
On a $H^0(G,B')=F$, et les classes de $\sigma\mapsto c'_\sigma$ et
$\sigma\mapsto Uc'_\sigma-b_\sigma$ dans $H^1(G,B')$ sont lin\'eairement ind\'ependantes sur $F$,
et donc $\dim_F H^1(G,B')\geq 2$.
\end{lemm}
\begin{proof}
La d\'erivation $\frac{d}{dU}$ commute \`a $G$. On en d\'eduit que, si $H^0(G,B')$
contient un \'el\'ement de degr\'e~$\geq 1$, il contient un \'el\'ement de degr\'e~$1$,
et donc il existe $a_1,b_0\in B$, avec $a_1\neq 0$, tels que $a_1U+a_0$ soit fixe par $G$.
Alors $a_1$ est aussi fixe par $G$, et donc $a_1\in F$, et $a_1c_\sigma+(\sigma-1)a_0=0$,
ce qui implique que $\sigma\mapsto c_\sigma$ est un cobord contrairement aux hypoth\`eses.
On a donc bien $H^0(G,B')=F$.

Si $\lambda_1,\lambda_2\in F$ et $y\in B'$ sont tels que
$\lambda_1 c'_\sigma+\lambda_2(Uc'_\sigma-b_\sigma)=(\sigma-1)y$, on peut appliquer
$\frac{d}{du}$ deux fois, ce qui tue le membre de gauche. On en d\'eduit que
$y$ est de la forme $a_2U^2+a_1U+a_0$, avec $a_2\in F$, $a_1,a_0\in B$.
En n'appliquant $\frac{d}{du}$ qu'une fois, cela nous donne
$\lambda_2c'_\sigma=(\sigma-1)(2a_2U+a_1)$ 
et donc 
$\sigma\mapsto \lambda_2c'_\sigma-2a_2c_\sigma$ est un cobord dans~$B$.
Par hypoth\`ese, ceci implique $\lambda_2=a_2=0$, 
et donc $\lambda_1 c'_\sigma=(\sigma-1)(a_1U+a_0)$.
En comparant les termes en $U$ des deux c\^ot\'es, on en d\'eduit que $a_1\in F$,
et que $\sigma\mapsto \lambda_1 c'_\sigma-a_1c_\sigma$ est un cobord dans $B$.
Il en r\'esulte que $\lambda_1=a_1=0$.

Ceci permet de conclure.
\end{proof}

\Subsection{L'alg\`ebre $B_{\rm uni}$}
Soit $B$ une $\Q$-alg\`ebre topologique int\`egre, 
munie d'une action continue d'un groupe profini $G$.  
L'action de $G$ sur $B$ s'\'etend \`a ${\rm Fr}(B)$, et on suppose que:

$\bullet$ $F:=B^G={\rm Fr}(B)^G$
et $H^1(G,B)\to H^1(G,{\rm Fr}(B))$ est injective.

$\bullet$ 
$H^1(G,B)$ est de dimension finie sur $F$ et $H^i(G,B)=0$ si $i\geq 2$.

\begin{theo}\label{tue21}
Il existe, \`a isomorphisme $G$-\'equivariant de $B$-alg\`ebres pr\`es, 
une unique $B$-alg\`ebre de polyn\^omes\footnote{
En une infinit\'e de variables si $\dim_{B^G}H^1(G,B)\geq 2$.}
 $B_{\rm uni}$, munie d'une action continue\footnote{$B_{\rm uni}$
est la limite inductive de $B$-modules libres de type fini, et on la munit
de la topologie de la limite inductive.} de $G$, et v\'erifiant les conditions
suivantes:

$\bullet$ $B_{\rm uni}^G={\rm Fr}(B_{\rm uni})^G=F$ et $H^i(G,B_{\rm uni})=0$
si $i\geq 1$.

$\bullet$ $B_{\rm uni}$ est la limite inductive d'alg\`ebres $\Lambda_n$,
o\`u $\Lambda_0=B$, $\Lambda_{n+1}=\Lambda_n$ ou bien $\Lambda_{n+1}=\Lambda_n[x_{n+1}]$ 
avec\footnote{
{\og uni\fg} est
une abr\'eviation pour {\og universel\fg} ou {\og unipotent\fg}, au choix.}
 $\sigma(x_{n+1})-x_{n+1}\in\Lambda_n$
pour tout $\sigma\in G$.
\end{theo}
\begin{proof}
La construction de $B_{\rm uni}$ fait l'objet du \S\,\ref{tue221}; la preuve de son unicit\'e
est donn\'ee dans la rem.\,\ref{tu241}.
\end{proof}

\begin{rema}\label{tue22}
Pour la construction d'une alg\`ebre $B_{\rm uni}$ telle que $H^i(G,B_{\rm uni})=0$
si $i\geq 1$, la premi\`ere condition ci-dessus sur $B$
peut \^etre remplac\'ee par {\og $F:=B^G$ est un corps\fg} (impliqu\'ee par
$B^G={\rm Fr}(B)^G$). Par contre, cette condition est utilis\'ee
crucialement pour prouver l'unicit\'e de $B_{\rm uni}$.
\end{rema}

\Subsection{Construction de $B_{\rm uni}$}\label{tue221}

L'alg\`ebre $B_{\rm uni}$ est une alg\`ebre filtr\'ee par le poids
($xy$ est de poids $i+j$ si $x$ est de poids $i$ et $y$ de poids $j$);
c'est la limite inductive d'alg\`ebres de polyn\^omes $B_n$ munies
d'une filtration $B_n^{\leq k}$ par le poids
(on a une suite exacte
$0\to B_n^{\leq k-1}\to B_n^{\leq k}\to {\rm Gr}^kB_n\to 0$) et d'une action de $G$
telle que $B_n^G=F$, et
qui laisse stable la filtration par les $B_n^{\leq k}$.  Les $B_n$ sont
construites par r\'ecurrence, de la mani\`ere suivante:

$\bullet$ $B_0=B$.

$\bullet$ $r_n=\dim_{F} H^1(G,B_{n-1}^{\leq n-1})-
\dim_{F}{\rm Ker}\big(H^1(G,B_{n-1}^{\leq n-1})\to H^1(G,B_{n-1}^{\leq n})\big)$.

$\bullet$ $\sigma\mapsto c^{(n,i)}_\sigma$, pour $1\leq i\leq r_n$,
sont des $1$-cocycles sur $G$ \`a valeurs dans $B_{n-1}^{\leq n-1}$,
dont les images dans $H^1(G,B_{n-1}^{\leq n-1})$ forment une base d'un suppl\'ementaire
de ${\rm Ker}\big(H^1(G,B_{n-1}^{\leq n-1})\to H^1(G,B_{n-1}^{\leq n})\big)$.

$\bullet$ $B_{n}=B_{n-1}[X_1^{(n)},\dots,X_{r_n}^{(n)}]$, o\`u les $X_i^{(n)}$
sont de poids $n$, et $\sigma\in G$ agit via la formule
$\sigma(X_i^{(n)})=X_i^{(n)}+c^{(n,i)}_\sigma$, pour $1\leq i\leq r_n$,
de telle sorte que $H^1(G,B_{n-1}^{\leq n-1})\to H^1(G,B_{n}^{\leq n})$ est identiquement $0$.

Pour v\'erifier les propri\'et\'es annonc\'ees, 
il faut v\'erifier que $B_n^G=F$ pour tout $n$; c'est l'objet du lemme~\ref{tue0}
ci-dessous.

Notons que ${\rm Gr}^kB_n$ est un $B$-module libre et que, si $d_{n,k}$ d\'esigne
son rang, on a ${\rm Gr}^kB_n\simeq B^{d_{n,k}}$ en tant que $G$-module.
On en d\'eduit, par r\'ecurrence sur $k$, que $H^i(G,B_n^{\leq k})=0$ si $i\geq 2$, pour tous $n,k$,
et que $H^1(G,B_n^{\leq k})$ est de dimension finie sur $F$.

Par construction, l'image de $H^1(G,B_n^{\leq k})\to H^1(G,B_m)$ est $0$ si $m$ est assez grand
(i.e., $m\geq n+\sup(k-n+1,0)$).
Or un $i$-cocycle continu sur $G$, \`a valeurs dans $B_{\rm uni}$, est \`a valeurs
dans $B_n^{\leq k}$ pour $n$ et $k$ assez grands car $G$ est compact et $B_{\rm uni}$
est la limite inductive des $B_n^{\leq k}$.  Il en r\'esulte que $H^i(G,B_{\rm uni})=0$
pour $i\geq 1$.

\begin{lemm}\label{tue0}
{\rm (i)} $B_n^G=F$.

{\rm (ii)} Si $x\in B_n$ et $(\sigma-1)\cdot x\in B_n^{\leq k}$ pour tout $\sigma\in G$, alors
$x\in B_n^{\leq k+1}$.
\end{lemm}
\begin{proof}
Pour $n=0$, les deux \'enonc\'es sont imm\'ediats.  Soit $n\geq 1$, et supposons que les
\'enonc\'es sont vrais pour $n-1$.  

Supposons que le (i) soit faux pour $n$.
Il existe donc $P=\sum_{\bf j}a_{\bf j}X^{\bf j}$, o\`u $X=(X^{(n)}_1,\dots,X^{(n)}_{r_n})$,
fixe par $G$ et de degr\'e total $d\geq 1$
en les $X_i^{(n)}$.
Mais, $\partial^{(n)}_i:=\frac{\partial}{\partial X_i^{(n)}}$ (d\'erivation $B_{n-1}$-lin\'eaire), pour $1\leq i\leq r_n$,
commute \`a l'action de $G$. Il en r\'esulte que $\partial^{(n)}_{i_1}\cdots\partial^{(n)}_{i_{d-1}}P\in B_n^G$
pour tout choix de $\partial^{(n)}_{i_1},\dots,\partial^{(n)}_{i_{d-1}}$. On peut donc supposer $d=1$,
i.e. $P=a_0+\sum_{i=1}^{r_n}b_iX^{(n)}_i$. Une comparaison des coefficients de
$P$ et de $\sigma(P)=\sigma(a_0)+\sum_{i=1}^{r_n}\sigma(b_i)(X_i^{(n)}+c_\sigma^{(n,i)})$
montre que $b_i\in B_{n-1}^G=F$, pour tout $i$, et que $\sum_{i=1}^{r_n}b_i c_\sigma^{(n,i)}=a_0-\sigma(a_0)$.

Utilisant le (ii) pour $n-1$, on en d\'eduit que $a_0\in B_{n-1}^{\leq n}$, ce qui contredit l'hypoth\`ese
selon laquelle les images des classes des $\sigma\mapsto c^{(n,i)}_\sigma$ dans $H^1(G,B_{n-1}^{\leq n})$
 forment une famille libre sur $F$.  On en d\'eduit que $B_n^G=F$.

On va montrer le (ii) pour $n$ par r\'ecurrence sur $k$.  Comme $B_n^{\leq k}=B_{n-1}^{\leq k}$ si
$k\leq n-1$, et comme $B_n^G=F\subset B_n^{\leq 0}$, l'\'enonc\'e est vrai pour $k\leq n-1$.
Si $k\geq n$, on peut appliquer $\partial^{(n)}_i$, et la condition
$(\sigma-1)\cdot x\in B_n^{\leq k}$ entra\^{\i}ne $(\sigma-1)(\partial_i^{(n)}x)\in B_n^{\leq k-n}$
pour $1\leq i\leq r_n$.  L'hypoth\`ese de r\'ecurrence
implique donc que $\partial_i^{(n)}x\in B_n^{\leq k-n+1}$, pour tout~$i$.

On peut \'ecrire $x$ sous la forme $x=a_0+x'$,
o\`u $a_0\in B_{n-1}$ et $x'=\sum_{{\bf j}\neq 0}a_{\bf j}X^{\bf j}$, 
et ce qui pr\'ec\`ede montre que $a_{\bf j}X^{\bf j}\in B_n^{\leq k+1}$ si ${\bf j}\neq 0$.
Il s'ensuit que le terme constant de $\sigma(x')$ appartient \`a $B_{n-1}^{\leq k}$
car passer de $X^{(n)}_i$ \`a $c^{(n,i)}_\sigma$ baisse le poids de $1$.
Comme celui de $\sigma(x)-x$ appartient \`a $B_{n-1}^{\leq k}$ puisque $\sigma(x)-x\in B_n^{\leq k}$,
on en d\'eduit que $\sigma(a_0)-a_0\in B_{n-1}^{\leq k}$ pour tout $\sigma\in G$,
et donc $a_0\in B_{n-1}^{k+1}$ par l'hypoth\`ese de r\'ecurrence.
Il en r\'esulte que $x\in B_n^{\leq k+1}$ ce que l'on voulait.
\end{proof}
\begin{rema} \label{tue23}
(i) Supposons\footnote{Si $\dim_F H^1(G,B)=0$ ou $1$, $X$ est fini et on peut prendre une suite 
stationnaire pour les $\Lambda_n$.}
 $\dim_F H^1(G,B)\geq 2$.  Soit $X=\{(n,i), \ n\geq 1,\ 1\leq i\leq r_n\}$.
Choisissons une bijection $\iota:X\to\N$ telle que $\iota(n,i)>\iota(m,j)$ si $n> m$,
et posons $X^{(n)}_i=Y_{\iota(n,i)}$.  Si $\Lambda_n=B[Y_1,\dots,Y_n]$, alors
$B_{\rm uni}$ est la limite inductive des $\Lambda_n$, chaque
$\Lambda_n$ est stable par $G$ et $\sigma(Y_n)-Y_n\in\Lambda_{n-1}$ pour tout
$\sigma\in G$.
La $B$-alg\`ebre $B_{\rm uni}$ v\'erifie donc toutes les conditions du th.\,\ref{tue21}.

(ii) $B_{\rm uni}$ est aussi la limite inductive des $B_{\rm uni}^{\leq n}=B_n^{\leq n}$,
et on a 
$$B_{\rm uni}^{\leq n}=B_{n-1}^{\leq n}\oplus\big(\oplus_{i=1}^{r_n}B\cdot X_i^{(n)}\big),
\quad
B_{\rm uni}^{\leq n}/B_{\rm uni}^{\leq n-1}\simeq B\otimes_{F}H^1(G,B)^{\otimes^n}
$$
On en d\'eduit que, si $\dim_{F}H^1(G,B)=1$, on a $B_{\rm uni}=B[X_1^{(1)}]$ 
(et si $H^1(G,B)=0$, alors $B_{\rm uni}=B$), mais si
$\dim_{F}H^1(G,B)\geq 2$, alors $B_{\rm uni}$ est une $B$-alg\`ebre de polyn\^omes
en une infinit\'e de variables.

Posons $W:=H^1(G,B)$ et $Z_n:=B_{\rm uni}^{\leq n}/B_{\rm uni}^{\leq n-1}$.
On prouve que $Z_n\simeq B\otimes_{F}W^{\otimes^n}$
par r\'ecurrence sur $n$; le cas $n=1$ \'etant imm\'ediat.  

Pour passer de $n-1$ \`a $n$,
on utilise la suite de cohomologie associ\'ee \`a la suite exacte
$0\to B_{\rm uni}^{\leq n-1}\to B_{\rm uni}^{\leq n}\to Z_n\to 0$.
Comme $H^1(G,B_{\rm uni}^{\leq n-1})\to H^1(G,B_{\rm uni}^{\leq n})$ est identiquement~$0$ par
construction, cette suite fournit des isomorphismes
$H^0(G,Z_n)\simeq H^1(G,B_{\rm uni}^{\leq n-1})$
et $H^1(G,B_{\rm uni}^{\leq n})\simeq H^1(G,Z_n)$.
L'hypoth\`ese de r\'ecurrence implique que $H^1(G,B_{\rm uni}^{\leq n-1})\simeq
H^1(G,B\otimes W^{\otimes n-1})=W^{\otimes n}$, et donc
$Z_n
=B\otimes
H^0(G,Z_n) \simeq B\otimes W^{\otimes n}$.
Ceci permet de conclure.

(iii) L'alg\`ebre $B_{\rm uni}$ est un cas particulier d'alg\`ebre de m\'elange 
({\og shuffle algebra\fg} en anglais). Voir l'appendice pour plus de d\'etails.
\end{rema}

\begin{exem}
On peut appliquer le th.\,\ref{tue21} (et le (ii) de la rem.\,\ref{tue23})
 \`a $C$ et \`a $\bdr^+$ pour en conclure
que 
$$C_{\rm uni}=C[\log t]
\quad{\rm et}\quad
(\bdr^+)_{\rm uni}=\bpdr^+$$
et donc retrouver la prop.\,\ref{gopal2} et le th.\,\ref{gopal4}.
\end{exem}

\Subsection{Plongement dans un $G$-anneau}\label{sect7}
Soient $L$ un corps de caract\'eristique $0$, 
$G$ un groupe agissant sur $L$, $\sigma\mapsto c_\sigma$ un $1$-cocycle
sur $G$ \`a valeurs dans $L$ dont la classe $c$ dans $H^1(G,L)$ est $\neq 0$.
On fait agir $G$ sur $L(X)$ en imposant $\sigma(X)=X+c_\sigma$, pour tout $\sigma\in G$.
\begin{lemm}\label{tue1}
{\rm (i)} On a 
$$H^0(G,L(X))=L^G, \quad {\rm Ker}(H^1(G,L)\to H^1(G,L(X)))=L^G\cdot c$$

{\rm (ii)} Si $L={\rm Fr}(A)$, o\`u $A$ est stable par $G$, si $\sigma\mapsto c_\sigma$ est
\`a valeurs dans $A$, et si $A^G=L^G$ et 
$H^1(G,A)\to H^1(G,L)$
est injectif, alors $H^1(G,A[X])\to H^1(G,L(X))$ est injectif.
\end{lemm}
\begin{proof}
Soit $\sigma\mapsto c'_\sigma$ un $1$-cocycle
sur $G$ \`a valeurs dans $L$.  On suppose que $\sigma\mapsto c'_\sigma$ se trivialise
dans $H^1(G,L(X))$.  Il existe donc $P,Q\in L[X]$, premiers entre eux, $Q$ unitaire,
tels que $c'_\sigma=\frac{P^\sigma(X+c_\sigma)}{Q^\sigma(X+c_\sigma)}-\frac{P(X)}{Q(X)}$
pour tout $\sigma\in G$.  

Posons $P_\sigma(X)=P^\sigma(X+c_\sigma)$ et $Q_\sigma(X)=Q^\sigma(X+c_\sigma)$. La relation
ci-dessus implique que $QQ_\sigma$ divise $P_\sigma Q- PQ_\sigma$, et comme $Q$ est premier \`a $P$,
cela implique $Q\mid Q_\sigma$; de m\^eme $Q_\sigma$ est premier \`a $P_\sigma$ et $Q_\sigma\mid Q$.
Comme $Q$ et $Q_\sigma$ sont unitaires,
on a $Q^\sigma(X+c_\sigma)=Q(X)$ pour tout $\sigma\in G$.
En comparant les termes
de degr\'e $e-1$, o\`u $e=\deg Q$, on obtient $e c_\sigma+\sigma(b_{e-1})=b_{e-1}$, ce
qui implique que $\sigma\mapsto c_\sigma$ est un cobord,
contrairement \`a l'hypoth\`ese.  Il s'ensuit
que $Q=1$.

\'Ecrivons $P=a_dX^d+a_{d-1}X^{d-1}+\cdots+a_0$.  Une comparaison des termes
de degr\'e $d$ implique $a_d\in L^G$, et celle des termes de degr\'e $d-1$ donne
$d a_d c_\sigma+\sigma(a_{d-1})-a_{d-1}=0$ si $d\geq 2$,
$a_1 c_\sigma+\sigma(a_{0})-a_{0}=c'_\sigma$, si $d\leq 1$.

$\bullet$ Pour $c'_\sigma=0$, cela implique que $\sigma\mapsto c_\sigma$ est un cobord
si $d\geq 1$,
contrairement \`a l'hypoth\`ese, et donc
$H^0(G,L(X))=L^G$.

$\bullet$ Dans le cas g\'en\'eral, cela implique que $d\leq 1$ et 
que $c'_\sigma-a_1 c_\sigma$
est un cobord.

Ceci prouve le (i).  Les arguments pour prouver le (ii) sont assez similaires:
si $P,Q\in L[X]$ sont premiers entre eux, avec $Q$ unitaire, et si
$\frac{P^\sigma(X+c_\sigma)}{Q^\sigma(X+c_\sigma)}-\frac{P(X)}{Q(X)}\in A[X]$ pour
tout $\sigma\in G$, les arguments ci-dessus prouvent que $Q=1$.
On est donc ramen\'e \`a prouver que, si $P^\sigma(X+c_\sigma)-P(X)\in A[X]$
pour tout $\sigma\in G$, alors $P\in A[X]$. On peut soustraire \`a $P$ un \'el\'ement
de $A[X]$ sans changer l'\'enonc\'e. Si l'\'enonc\'e est faux,
on peut donc supposer que $P(X)=a_dX^d+\cdots+a_0$,
o\`u $a_d\in L\moins A$.  Mais alors $\sigma\to \sigma(a_d)-a_d$ est un $1$-cocycle
sur $G$, \`a valeurs dans $A$, dont la classe dans $H^1(G,L)$ est $0$.  Comme
$H^1(G,A)\to H^1(G,L)$ est injectif, il existe $a'_d\in A$ tel que
$\sigma(a_d)-a_d=\sigma(a'_d)-a'_d$ pour tout $\sigma\in G$, et comme
$L^G=A^G$, on a $a'_d-a_d\in A$, et donc $a_d\in A$, contrairement \`a ce que l'on avait 
suppos\'e.
\end{proof}
Soit $\Lambda$ un anneau int\`egre muni d'une action de $G$, et d'une injection
$G$-\'equivariante $L\hookrightarrow \Lambda$. 

\begin{lemm}\label{tue2}
Soit $x\in\Lambda$ v\'erifiant $\sigma(x)-x=c_\sigma$ pour tout $\sigma\in G$.
Alors $x$ est transcendant sur $L$ et l'injection $L\hookrightarrow \Lambda$
s'\'etend en des injections $G$-\'equivariantes $L[X]\hookrightarrow \Lambda$, 
$L(X)\hookrightarrow {\rm Fr}(\Lambda)$, en envoyant $X$ sur $x$.
\end{lemm}
\begin{proof}
Supposons que $x$ est alg\'ebrique sur $L$.
Soit $P(X)=X^d+a_{d-1}X^{d-1}+\cdots+a_0$ le polyn\^ome minimal de $x$ sur $L$.
Par unicit\'e du polyn\^ome minimal, on a
$$P(X)=(X+c_\sigma)^d+\sigma(a_{d-1})(X+c_\sigma)^{d-1}+\cdots+\sigma(a_0)$$
 pour tout
$\sigma\in G$. L'identification des termes de degr\'e $d-1$ nous donne
$dc_\sigma+\sigma(a_{d-1})-a_{d-1}=0$, ce qui prouve que $\sigma\mapsto c_\sigma$
est un cobord, contrairement \`a l'hypoth\`ese.
\end{proof}

Soit $\Lambda$ un anneau int\`egre, muni d'une action de $G$.  On suppose que
$H^1(G,\Lambda)=0$.
\begin{lemm}\label{tue24}
Si $\iota:B\to \Lambda$ est une injection $G$-\'equivariante, alors $\iota$
se prolonge en des injections $G$-\'equivariantes $B_{\rm uni}\hookrightarrow \Lambda$
et ${\rm Fr}(B_{\rm uni})\hookrightarrow {\rm Fr}(\Lambda)$.
\end{lemm}
\begin{proof}
\'Ecrivons $B_{\rm uni}$ comme la limite inductive des $\Lambda_n$ comme dans la
rem.\,\ref{tue23}.  Il suffit de prouver l'\'enonc\'e pour $\Lambda_n$ \`a la place
de $B_{\rm uni}$, ce qui se fait par r\'ecurrence sur $n$, en utilisant les lemmes~\ref{tue1}
et~\ref{tue2}.
\end{proof}
\begin{rema}\label{tu241}
Si $B_1,B_2$ sont deux $B$-alg\`ebres satisfaisant les conclusions du th.\,\ref{tue21},
le lemme~\ref{tue24} fournit des injections 
$G$-\'equivariantes de $B$-alg\`ebres $\iota_1:B_1\to B_2$ et $\iota_2:B_2\to B_1$.
La compos\'ee fournit une injection $G$-\'equivariante de $B$-alg\`ebres 
$\iota:B_1\to B_1$; pour prouver
l'unicit\'e \`a isomorphisme pr\`es de $B_{\rm uni}$ il suffit de prouver
qu'une telle injection $G$-\'equivariante est automatiquement un isomorphisme.
Pour cela, il suffit de prouver que $\iota$ induit un automorphisme de $B_{\rm uni}^{\leq n}$
pour tout $n$.  

Pour $n=0$, cela r\'esulte de ce que $\iota$ est l'identit\'e sur $B$, par hypoth\`ese.
Si $\iota$ induit un automorphisme de $B_{\rm uni}^{\leq i}=B_{i}^{\leq i}$, pour $i\leq n-1$,
il induit aussi un automorphisme de $B_{n-1}$ respectant la filtration par le poids,
ainsi que des automorphismes de $H^1(G,B_{n-1}^{\leq n-1})$ et du noyau
de $H^1(G,B_{n-1}^{\leq n-1})\to H^1(G,B_{n-1}^{\leq n})$. Il s'ensuit que
les images des $\sigma\mapsto\iota(c_\sigma^{(n,i)})$ forment une base d'un suppl\'ementaire
de ce noyau dans $H^1(G,B_{n-1}^{\leq n-1})$. Comme $B_{\rm uni}^G=F\subset B_{n-1}^{\leq n}$,
on en d\'eduit que $\iota(X_i^{(n)})\in B_{\rm uni}^{\leq n}$, et que
$\iota$ induit un automorphisme de $B_{\rm uni}^{\leq n}/B_{n-1}^{\leq n}$,
et donc aussi de $B_{\rm uni}^{\leq n}$.
\end{rema}

\section{La conjecture $C_{\rm st}$ pour la cohomologie \`a support compact}\label{sect8}
L'anneau $\bcris^+$ (\`a partir duquel $\bst$ est construit) peut \^etre remplac\'e par 
$\B_{\rm rig}^+:=\cap_{n\geq 0}\varphi^n(\bcris^+)$ dans la formulation de la conj.\,\ref{cst1}.
Mais $\B_{\rm rig}^+$ peut s'interpr\'eter comme
 l'anneau des fonctions analytiques sur $({\rm Spa}\,\ainf)\moins\{p=0\}$,
et on peut aussi le remplacer par $\B=\O(Y_{\rm FF})$, 
o\`u $Y_{\rm FF}=({\rm Spa}\,\ainf)\moins(\{p=0\}\cup\{\tilde p=0\})$
car tous ces anneaux ont les m\^emes sous-$\Q_p$-modules de rang fini stables par $\varphi$.
L'int\'er\^et de remplacer $\bcris^+$ par $\B$ est que la cohomologie galoisienne de ce dernier
est nettement plus facile \`a d\'ecrire.
\Subsection{Cohomologie galoisienne de $\B$}\label{sect9} 
Soit donc $Y_{\rm FF}$ la courbe de Fargues-Fontaine obtenue en retirant de ${\rm Spa}(\ainf)$
les diviseurs $p=0$ et $\tilde p=0$, o\`u $\tilde p=[p^\flat]$.
Si $I$ est un intervalle de $\R_+^\dual$, soit $\B^I=\O(Y_{\rm FF}^I)$, o\`u
$Y_{\rm FF}^I$ est le lieu des $v$ v\'erifiant $v(\tilde p)/v(p)\in I$.
Si $I=[r,s]$ avec $r,s\in\Q$, alors $\B^I=\A^{[r,s]}[\frac{1}{p}]$, o\`u $\A^{[r,s]}$ est le
compl\'et\'e $p$-adique 
de $\ainf[\frac{\tilde p^{1/r}}{p},\frac{p}{\tilde p^{1/s}}]$.
Soit 
$$\B:=\B^{]0,\infty[}=\O(Y_{\rm FF})=\varprojlim\B^{[r,s]}.$$  
On note $x_{\rm dR}\in Y_{\rm FF}$ le point d\'efini par $\tilde p=p$; le compl\'et\'e de
l'anneau local en ce point est alors $\bdr^+$, ce qui fournit un morphisme
injectif $\B\hookrightarrow\bdr^+$ qui est $G_K$-\'equivariant.
\begin{lemm} \label{gopal8}
\cite[(2.13)]{bosco}.
On a des isomorphismes
$$\B/t\B=\prod\nolimits_{n\in\Z}C\quad{\rm et}\quad \B^{]0,r]}/t\B^{]0,r]}=\prod\nolimits_{p^nr\geq 1}C$$
\end{lemm}
\begin{proof}
Si $r,s\in\Q_+^\dual$, il r\'esulte de~\cite[th.2.5.1]{FF} et~\cite[th.2.4.10]{FF}
que $\B^{[r,s]}/t\B^{[r,s]}=\prod_{y\in{\rm Div}(t)\cap Y_{\rm FF}^{[r,s]}}\bdr^+(C_y)/{\rm Fil}^{v_y(t)}$.
Or les z\'eros de $t$ sont l'ensemble des translat\'es de $x_{\rm dR}$ par $\varphi^\Z$, chacun
avec multiplicit\'e~$1$. Comme les corps r\'esiduels en tous ces translat\'es sont, naturellement,
isomorphes \`a $C$, on obtient $\B^{[r,s]}/t\B^{[r,s]}=\prod_{r\leq p^n\leq s}C$.
Le r\'esultat s'en d\'eduit
en utilisant la nullit\'e de ${\rm R}^1\varprojlim t\B^{[r,s]}$ (qui r\'esulte du crit\`ere
de Mittag-Leffler topologique).
\end{proof}

On d\'efinit $\log\tilde p$ comme un \'el\'ement transcendant sur $\B$, v\'erifiant
$$\varphi(\log\tilde p)=p\,\log\tilde p
\quad{\rm et}\quad
\sigma(\log\tilde p)=\log\tilde p+{\rm Kum}_p(\sigma)\,t$$
o\`u ${\rm Kum}_p:G_K\to \Z_p(\chi_{\rm cycl})$ est le cocycle de Kummer associ\'e \`a $p$ et
$t$ est le $2i\pi$ $p$-adique de Fontaine. 
On pose alors:
$$\bff^+:=\B[\log\tilde p],\quad\bff:=\bff^+[\tfrac{1}{t}]$$

On \'etend l'injection 
$\B\hookrightarrow\bdr^+$ en un morphisme d'anneaux
$K\otimes_F\bff^+\to\bdr^+$ en envoyant $\log\tilde p$ sur $\log\frac{\tilde p}{p}$
(la s\'erie d\'efinissant $\log\frac{\tilde p}{p}$ converge puisque
$\frac{\tilde p}{p}-1$ est dans l'id\'eal maximal de $\bdr^+$).
Ce morphisme est $G_K$-\'equivariant.

\begin{prop} {\rm (\cite[prop.\,10.3.15]{FF})}\label{gopal5}
Le morphisme $K\otimes_F\bff^+\to\bdr^+$ ci-dessus est injectif.
\end{prop}

\begin{lemm}\label{gopal9}
L'application naturelle $H^i(G_K,\B)\to H^i(G_K,\B[\frac{1}{t}])$ est un isomorphisme.
\end{lemm}
\begin{proof}
Comme $G_K$ est compact, $H^i(G_K,\B[\frac{1}{t}])=\varinjlim_j H^i(G_K, t^{-j}\B)$
et il suffit de prouver que, si $j\geq 0$, alors
$H^i(G_K, t^{-j}\B)\to H^i(G_K, t^{-j-1}\B)$ est un isomorphisme.
Il r\'esulte du lemme~\ref{gopal8} que
l'on a une suite exacte 
$$0\to t^{-j}\B\to t^{-j-1}\B\to\prod\nolimits_{n\in\Z}C(-j-1)\to 0$$
et le r\'esultat est une cons\'equence de ce que
$H^i(G_K,\prod_{n\in\Z}C(-j-1))=0$ pour tout~$i$,
d'apr\`es la prop.\,\ref{gopal1}.
\end{proof}

\begin{lemm}\label{gopal10}
L'image de $H^1(G_K,t\B)$ dans $H^1(G_K,\B)$ 
est un $F$-espace de dimension $1$ engendr\'e par la classe de 
$\sigma\mapsto(\sigma-1)\cdot\log\tilde p$.
\end{lemm}
\begin{proof}
Soit $\sigma\mapsto c_\sigma$ un $1$-cocycle \`a valeurs dans $t\B$.
Fixons $r\in]1,p[$; alors $\B\subset \B^{]0,r]}$. Choisissons $s\in]1,r[$.  
D'apr\`es \cite[prop.\,10.7]{drpst}, 
appliqu\'ee \`a $\varphi^n(c_\sigma)$ (qui v\'erifie le (ii) de la proposition
en question puisqu'il est \`a valeurs dans $t\bdr^+$), il existe
$a'_n\in F$ et $b'_n\in \B^{]0,s]}$ tels que $\varphi^n(c_\sigma)=(\sigma-1)(a'_n\log\tilde p+b'_n)$.
On a alors $c_\sigma=(\sigma-1)(a_n\log\tilde p+b_n)$, avec $a_n=p^{-n}\varphi^{-n}(a'_n)\in F$
et $b_n=\varphi^{-n}(b'_n)\in\B^{]0,p^ns]}$. 
Par ailleurs, comme $H^0(G_K,\B^{]0,s]}\oplus F\log\tilde p)=F$, on a $a_{n+1}=a_n$ pour tout $n$ et
$b_{n+1}-b_n\in F$, pour tout $n$. On peut donc modifier $b_n$ par un \'el\'ement de $F$ de telle sorte
que $b_{n+1}=b_n:=b$ pour tout $n$, et alors $b\in\cap_{n\geq 0}\B^{]0,p^ns]}=\B$.

Le r\'esultat s'en d\'eduit.
\end{proof}

\begin{lemm}\label{gopal11}
$H^i(G_K,t^j\B)=0$ pour tout $j\in\Z$ et tout $i\geq 2$.
\end{lemm}
\begin{proof}
Soient $H_K:={\rm Ker}\,\chi_{\rm cycl}$ et $\Gamma_K:=G_K/H_K={\rm Gal}(K(\bmu_{p^\infty})/K)$.
Par descente presque \'etale, on prouve que $H^i(H_K,t^j\B)=0$ pour tout $i\geq 1$ et tout~$j$.
Donc l'inflation $H^i(\Gamma_K,H^0(H_K,t^j\B))\to H^i(G_K,t^j\B)$ est un isomorphisme
pour tout $i\geq 1$ et tout $j$. Le r\'esultat est donc
une cons\'equence de ce que $\Gamma_K$ est de dimension
cohomologique $1$.
\end{proof}

\begin{prop}\label{gopal12}
On a une suite exacte
$$0\to F\cdot{\rm Kum}_p\to H^1(G_K,\B)\to (\prod\nolimits_{n\in\Z}K)\cdot\log\chi_{\rm cycl}\to 0$$
\end{prop}
\begin{proof}
Cela r\'esulte de
la suite exacte longue de cohomologie associ\'ee \`a la suite exacte
$0\to t\B\to \B\to \prod_{n\in\Z}C\to 0$, de la nullit\'e de $H^2(G_K,t\B)$ (lemme~\ref{gopal11},
pour la surjectivit\'e \`a droite), du lemme~\ref{gopal10} et de la prop.\,\ref{gopal1}
dont on tire l'identit\'e $H^1(G_K,\prod_{n\in\Z}C)=(\prod_{n\in\Z}K)\cdot\log\chi_{\rm cycl}$.
\end{proof}

\begin{rema}\label{gopal12.1}
(i) Les m\^emes arguments montrent que, si $I\subset]0,\infty[$ est un intervalle ouvert, alors:

$\bullet$ L'application naturelle $H^i(G_K,\B^I)\to H^i(G_K,t^{-j}\B^I)$ est un isomorphisme
pour tous $i\geq 0$ et $j\geq 0$.

$\bullet$
Si $\Z(I)$ est l'ensemble des $n\in\Z$, tels que $p^{-n}\in I$,
on a une suite exacte
$$0\to F\cdot{\rm Kum}_p\to H^1(G_K,\B^I)\to (\prod\nolimits_{n\in\Z(I)}K)\cdot\log\chi_{\rm cycl}\to 0$$
Il en r\'esulte que, si $I$ est \`a extr\'emit\'es dans $\Q^\dual$,
alors $H^1(G_K,\B^I)$ est un $F$-module de rang fini.

Par ailleurs $H^i(G_K,\B^I)=0$ si $i\geq 2$, car une descente presque \'etale montre
que ce groupe est isomorphe \`a $H^i(\Gamma_K,(\B^I)^{H_K})$, o\`u $\Gamma_K={\rm Gal}(K(\bmu_{p^\infty})/K)$
et $H_K={\rm Gal}(\overline K/K(\bmu_{p^\infty}))$, et $\Gamma_K$ a dimension cohomologique~$1$.

(ii) On d\'eduit de la description de $H^1(G_K,\B^I)$ ci-dessus que, si $I\subset J\subset]0,\infty[$
sont des intervalles ouverts, alors $H^1(G_K,\B^J)\to H^1(G_K,\B^I)$ est surjective.

(iii) Si $I$ n'est pas ouvert (i.e.~ferm\'e ou ferm\'e d'un seul c\^ot\'e),
alors $H^1(G_K,\B^I)$ est de rang infini sur $F$ et n'est pas s\'epar\'e.
\end{rema}

\Subsection{Les $\B^I_{\rm uni}$}\label{sect10}

Le r\'esultat suivant permet de comparer
$\B^I_{\rm uni}$ \`a $\bpdr^+$ ou $\B^J_{\rm uni}$ (prop.\,\ref{tue10} ci-dessous). 

\begin{lemm}\label{tue4}
Soit $I\subset ]0,+\infty[$ un intervalle \`a extr\'emit\'es dans $\Q^\dual$.
Alors
$H^0(G_K,\B^I)=H^0(G_K,{\rm Fr}(\B^I))$ et
$H^1(G_K,\B^I)\to H^1(G_K,{\rm Fr}(\B^I))$ est injective.
\end{lemm}
\begin{proof}
Soit $\sigma\mapsto c_\sigma$ un $1$-cocycle \`a valeurs dans $\B^I$. On suppose
que ce cocycle est un cobord dans ${\rm Fr}(\B^I)$, et donc qu'il existe $c\in
{\rm Fr}(\B^I)$ avec $c_\sigma=\sigma(c)-c$ pour tout $\sigma\in G_K$.
On cherche \`a prouver que $c\in\B^I$ (le cas $c_\sigma=0$ pour tout $\sigma$
implique le r\'esultat pour les $H^0$).

Si $I$ est ferm\'e (resp.~n'est pas ferm\'e), 
l'anneau $\B^I$ est principal (resp.~de B\'ezout), et on peut donc \'ecrire $c=\frac{a}{b}$
avec $a$ et $b$ premiers entre eux.
Comme $\sigma(c)-c\in \B^I$, on en d\'eduit
que les diviseurs de $\sigma(b)$ et de $b$ sont les m\^emes.
Il en r\'esulte que le diviseur de $b$ est stable par $G_K$, et donc 
est concentr\'e en les $\varphi^n(x_{\rm dR})$ appartenant \`a $Y_{\rm FF}^I$, et
cet ensemble est fini gr\^ace \`a l'hypoth\`ese faite sur $I$.

On peut d\'ecomposer la partie polaire en \'el\'ements simples.
Notons $\frac{[\alpha]}{\xi_n^m}$ le terme de plus haut degr\'e pour la partie
correspondant \`a $\varphi^n(x_{\rm dR})$. Comme ce terme dispara\^{\i}t dans
$\sigma(c)-c$, on doit avoir 
$\sigma(\frac{[\alpha]}{\xi_n^m})-\frac{[\alpha]}{\xi_n^m}=0$ mod $\frac{1}{\xi_n^{m-1}}\B^I$;
autrement dit, l'image de $\frac{[\alpha]}{\xi_n^m}$ dans $C(-m)$ est fixe
par $G_K$, 
ce qui est impossible, si $m\geq 1$, d'apr\`es Tate.
On en d\'eduit que $c$ n'a de p\^ole en aucun des $ \varphi^n(x_{\rm dR})$, et donc
$c\in\B^I$, ce qui permet de conclure.
\end{proof}

D'apr\`es la rem.\,\ref{gopal12.1}, 
si $I\subset ]0,+\infty[$ est un intervalle ouvert \`a extr\'emit\'es dans $\Q^\dual$,
alors  $\B^I$ satisfait les hypoth\`eses du th.\,\ref{tue21},
et donc on dispose de la $\B^I$-alg\`ebre $\B^I_{\rm uni}$.
\begin{prop}\label{tue10}
Soient $I,J$ des intervalles ouverts \`a extr\'emit\'es dans $\Q^\dual$.

{\rm (i)} Si $Y_{\rm FF}^I$ contient $ x_{\rm dR}$, l'injection $\B^I\hookrightarrow \bpdr^+$
s'\'etend en une injection $G_K$-\'equivariante $\B^I_{\rm uni}\hookrightarrow \bpdr^+$.

{\rm (ii)} Si $Y_{\rm FF}^I$ ne contient aucun $\varphi^n( x_{\rm dR})$, 
alors $\B^I_{\rm uni}=\B^I[\log\tilde p]$.

{\rm (iii)} Si $J\subset I$, l'injection $\B^I\hookrightarrow \B^J$
s'\'etend {\rm(de mani\`ere non unique)} en une injection $G_K$-\'equivariante 
$\iota_{I}^{J}:\B^I_{\rm uni}\hookrightarrow \B^J_{\rm uni}$,
et le sous-anneau $\iota_{I}^{J}(\B^I_{\rm uni})$ de $\B^J_{\rm uni}$
 ne d\'epend pas de $\iota_{I}^{J}$.

{\rm (iv)} Si $J= p^{-1}I$, l'isomorphisme $\varphi:\B^I\overset{\sim}{\to} \B^J$
s'\'etend {\rm(de mani\`ere non unique)} en un isomorphisme $G_K$-\'equivariant 
$\varphi_I^J:\B^I_{\rm uni}\overset{\sim}{\to} \B^J_{\rm uni}$.
\end{prop}
\begin{proof}
Tout r\'esulte du th.\,\ref{tue21},
et des lemmes~\ref{tue4} et~\ref{tue24}, \`a part pour le fait que les images
ne d\'ependent pas des choix des extensions.  Ceci est une cons\'equence du lemme~\ref{tue107}
ci-dessous.
\end{proof}

\begin{lemm}\label{tue107}
Si $u:\B^I\to\B^J$ est une injection $G_K$-\'equivariante, et
si $\tilde u:\B_{\rm uni}^I\to\B_{\rm uni}^J$ est une extension
$G_K$-\'equivariante de $u$, alors l'image de $\tilde u$ ne d\'epend pas
de $\tilde u$ et ne d\'epend que de l'image de $u$.
\end{lemm}
\begin{proof}
Ceci se prouve
comme l'unicit\'e de $B_{\rm uni}$ (cf.~rem.\,\ref{tu241}), en montrant,
par r\'ecurrence sur $n$, que $\tilde u(\B^I_n)$
(et m\^eme l'image de la filtration)
ne d\'epend pas du choix
de l'extension: en effet, une extension \`a $\B^I_{n-1}$ \'etant choisie,
les $\tilde u(X_i^{(n)})$ sont bien d\'efinis
\`a addition pr\`es d'\'el\'ements de $(\B^J_{\rm uni})^{G_K}=F$,
et comme $F\subset u(\B^I)$, l'alg\`ebre engendr\'ee
(et m\^eme l'image de $\B_{\rm uni}^{I,\leq n}$) ne d\'epend pas de ces choix.
\end{proof}

\begin{rema}\label{tue112}
Le (ii) de la prop.\,\ref{tue10} implique, 
en particulier, que $\B^I[\log\tilde p]$ contient $\log t$ si
$Y_{\rm FF}^I$ ne contient aucun $\varphi^n(x_{\rm dR})$.  
L'hypoth\`ese \'equivaut \`a ce qu'il existe $n\in\Z$ tel que
$I\subset]\frac{1}{p^{n+1}},\frac{1}{p^n}[$. En \'ecrivant 
$t$ sous la forme 
$\frac{1}{p^n}\log(1+\varphi^n(\pi))=\frac{1}{p^n}\varphi^n(\pi)E(\varphi^n(\pi))$
o\`u $\pi=[\epsilon]-1$ et 
$E(X)=X^{-1}\log(1+X)=1-\tfrac{X}{2}+\tfrac{X^2}{3}+\cdots$,
et en utilisant la branche de $\log$ telle que $\log p=0$, cela permet d'\'ecrire
$\log t$
sous la forme 
$$\log t=\log(E(\varphi^n(\pi)))
+\log(\varphi^n(\pi)/[\epsilon-1]^{p^n})+
\tfrac{p^n}{2(p-1)}\log([\epsilon-1]^{2(p-1)}/\tilde p^{2p})+
\tfrac{p^{n+1}}{p-1}\log\tilde p$$
Nous laissons au lecteur le plaisir de v\'erifier que le $1^{\rm er}$ terme converge si
$\frac{v(\tilde p)}{v(p)}>\frac{1}{p^{n+1}}$, le $2^{\rm nd}$
si $\frac{v(\tilde p)}{v(p)}<\frac{1}{p^{n}}$, et le $3^{{\text{\`eme}}}$ partout car
$u:=([\epsilon-1]^{p-1}/\tilde p^p)^2$ est une unit\'e de~$\ainf$ v\'erifiant $v(u-1)>0$
pour tout $v$ (une traduction de $v_p((\zeta_p-1)^{p-1}+p)\geq 1+\frac{1}{p-1}$),
et que l'on a bien $(\sigma-1)\cdot\log t=\log(\chi_{\rm cycl}(\sigma))$ avec cette d\'efinition
de $\log t$. 
\end{rema}

\begin{rema}\label{ass1}
 (i) Posons 
$W_I:=H^1(G_K,\mathbf B^I)$, $P_I:=\mathbf B^I\otimes_F\operatorname{Sh}_F(W_I),$
Alors, d'apr\`es le corollaire \ref{kicius1}, 
on a un isomorphisme filtr\'e \(G_K\)-\'equivariant non canonique
$\mathbf B^I_{\mathrm{uni}}\simeq P_I,$
o\`u l'action sur le membre de droite est construite dans l'appendice.

 (ii) Pour \(J\subset I\), soient $\rho_{IJ}:W_I\to W_J$ l'application naturelle, et 
$ \mathcal U_I:=\operatorname{Spec}\operatorname{Sh}_F(W_I)$.
Le corollaire \ref{nogi1a} renforce la proposition \ref{tue10} : 
pour deux prolongements \'equivariants quelconques
$
\iota_1,\iota_2:P_I\to  P_J $
de $\iota:\B^I\to\B^J$, il existe un unique \(v\in\mathcal U_I(F)\) tel que
$\iota_1=\iota_2\circ h_v.$
De plus, \(h_v\) pr\'eserve la filtration et agit trivialement sur le gradu\'e associ\'e:
i.e.~$ \iota_1(P_{I,\le n})=\iota_2(P_{I,\le n}) $
pour tout \(n\), et
${\rm gr}_n(\iota_i)= \iota\otimes\rho_{IJ}^{\otimes n}$.
\end{rema}

\begin{lemm}\label{tue90}
Soit $J$ tel que $Y_{\rm FF}^J$ ne contienne aucun $\varphi^n( x_{\rm dR})$,
et soit $N:\B^J_{\rm uni}\to \B^J_{\rm uni}$ la d\'eriv\'ee $\B^J$-lin\'eaire
telle que $N(\log\tilde p)=-1$. Si $J\subset I$, alors $N$ laisse stable
$\iota_{I}^J(\B^I_{\rm uni})$; plus pr\'ecis\'ement, 
$N$ envoie $\iota_{I}^J((\B^I_{\rm uni})^{\leq n})$ dans $\iota_{I}^J((\B^I_{\rm uni})^{\leq n-1})$.
\end{lemm}
\begin{proof}
D\'emontrons le r\'esultat plus pr\'ecis, par r\'ecurrence sur $n$.
Comme $N$ tue $\B^J$ et donc aussi $\iota(\B^I)$, le r\'esultat est vrai pour $n=0$.

Maintenant, $\B^I_{\rm uni}$ est la limite inductive des $\B^I_n$ comme
dans le (ii) de la rem.\,\ref{tue23}, et il suffit
de prouver que $N$ 
envoie $\iota_{I}^J((\B^I_n)^{\leq k})$ dans $\iota_{I}^J((\B^I_n)^{\leq k-1})$, pour tous
$k,n$. Comme $N$ est une d\'erivation, le r\'esultat pour $k=n$ implique
le r\'esultat pour tout $k$.

Notons simplement $\iota$ l'injection $\iota_I^J$ et $B_n^{\leq k}$ le $\B^I$-module
$(\B^I_n)^{\leq k}$.
On a $B_n^{\leq n}=B_{n-1}^{\leq n}\oplus\big(\oplus_{i=1}^{r_n}\B^IX_i^{(n)}\big)$,
et il suffit de prouver que $N(\iota(X_i^{(n)}))\in \iota(B_{n-1}^{\leq n-1})$.
Supposons le r\'esultat vrai pour $n-1$.
Comme $N$ commute \`a $G_K$, on a $(\sigma-1)\cdot N(\iota(X_i^{(n)}))=
N(\iota(c_\sigma^{(n,i)}))$,
o\`u $\sigma\mapsto c_\sigma^{(n,i)}$ est un $1$-cocycle \`a valeurs dans $B_{n-1}^{\leq n-1}$,
et donc $\sigma\mapsto N(\iota(c_\sigma^{(n,i)}))$ 
est un $1$-cocycle \`a valeurs dans $\iota(B_{n-2}^{\leq n-2})$ (avec la convention $B_{m}=0$,
si $m<0$), 
d'apr\`es l'hypoth\`ese de
r\'ecurrence.  

Comme $H^1(G_K,B_{n-2}^{\leq n-2})\to H^1(G_K,B_{n-1}^{\leq n-1})$
est identiquement~$0$,
il existe
$z_i^{(n)}\in \iota(B_{n-1}^{\leq n-1})$ tel que $(\sigma-1)z_i^{(n)}=
(\sigma-1)\cdot N(\iota(X_i^{(n)}))$ pour tout $\sigma$. Ceci implique
que $N(\iota(X_i^{(n)}))-z_i^{(n)}\in F$, et donc que 
$N(\iota(X_i^{(n)}))\in\iota(B_{n-1}^{\leq n-1})$, ce que l'on voulait.
\end{proof}
\Subsection{L'anneau $\bpff$}\label{sect11}
\Subsubsection{L'anneau $\Bpff{K}$}
Posons 
$$\B^{(I),\leq m}_{\rm uni}:=\cap_{J\supset I}\iota_{J}^I(\B^{J,\leq m}_{\rm uni}),
\qquad \B^{(I)}_{\rm uni}:=\varinjlim\nolimits_m \B^{(I),\leq m}_{\rm uni}$$
\begin{lemm}\label{tue1000}
On a ${\rm R}^1\varprojlim_J\iota_{J}^I(\B^{J,\leq m}_{\rm uni})=0$, pour tout $m\geq 0$.
\end{lemm}
\begin{proof}
Pour $m=0$, cela r\'esulte de ce que $\B^J$ est d'image dense dans $\B^I$ si $J\supset I$,
et donc que le syst\`eme v\'erifie la condition de Mittag-Leffler topologique.
Maintenant, pour tout $J$, 
on a une suite exacte 
$$0\to \B^{J,\leq m-1}_{\rm uni}\to \B^{J,\leq m}_{\rm uni}
\to \B^J\otimes_F W_J^{\otimes m}\to 0,\quad{\text{o\`u $W_J=H^1(G_K,\B^J)$.}}$$
Par ailleurs, si $J\supset I$, l'application naturelle $W_J\to W_I$ est surjective
 (cf.~(ii) de la rem.\,\ref{gopal12.1}),
et donc $W_J^{\otimes m}\to W_I^{\otimes m}$ aussi. Comme $\B_J\to\B_I$ est d'image dense,
on en d\'eduit que le syst\`eme des $\B^J\otimes_F W_J^{\otimes m}$ 
v\'erifie la condition de Mittag-Leffler topologique.  Il s'ensuit que son ${\rm R}^1\varprojlim$
est $0$, et une r\'ecurrence imm\'ediate permet d'en d\'eduire que
${\rm R}^1\varprojlim_J\iota_{J}^I(\B^{J,\leq m}_{\rm uni})=0$, pour tout $m$.
\end{proof}

\begin{rema}\label{tue1001}
Si $(I_n)_{n\geq 0}$ est un suite croissante d'intervalles ouverts de $]0,+\infty[$,
\`a extr\'emit\'es dans $\Q^\dual$, telle que $I_0=I$ et $\cup_n I_n=]0,+\infty[$, 
le lemme~\ref{tue1000}
se traduit par l'existence d'une suite exacte stricte:
$$0\to \B^{(I),\leq m}_{\rm uni}\to \prod\nolimits_n\B_{\rm uni}^{I_n,\leq m}
\to \prod\nolimits_n\B_{\rm uni}^{I_n,\leq m}\to 0$$
qui permet de calculer la cohomologie galoisienne de $\B^{(I),\leq m}_{\rm uni}$
\`a partir de celle des $\B_{\rm uni}^{I_n,\leq m}$.
\end{rema}

D'apr\`es la prop.\,\ref{tue10}, $\B^{(I)}_{\rm uni}\subset \B^I_{\rm uni}$ ne d\'epend pas du choix
des $\iota_{J}^{I}$. 
Ce qui pr\'ec\`ede s'applique \`a $I_1=]\frac{1}{p},1+\frac{1}{p}[$, 
$I_2=]\frac{1}{p},\frac{1}{p}+\frac{1}{p^2}[$ (remarquons que
$\B^{I_2}_{\rm uni}=\B^{I_2}[\log\tilde p]$ car 
$Y_{\rm FF}^{I_2}$ ne contient aucun $\varphi^n(x_{\rm dR})$). 
On d\'eduit de la prop.\,\ref{tue10}
que les injections
$\iota:\B^{I_1}\to\B^{I_2}$ et $\iota':\B^{p^{-1}I_1}\to\B^{I_2}$ se prolongent
 (de mani\`ere non unique) 
en des injections $\iota:\B_{\rm uni}^{I_1}\to\B_{\rm uni}^{I_2}$ et 
$\iota':\B_{\rm uni}^{p^{-1}I_1}\to\B_{\rm uni}^{I_2}$
et que ces injections
induisent des isomorphismes $\iota:\B_{\rm uni}^{(I_1)}\simeq \B_{\rm uni}^{(I_2)}$
et $\iota':\B_{\rm uni}^{(p^{-1}I_1)}\simeq \B_{\rm uni}^{(I_2)}$.
De m\^eme l'isomorphisme $\varphi:\B^{I_1}\to\B^{p^{-1}I_1}$ se prolonge
(de mani\`ere non unique) en des isomorphismes 
$\varphi:\B_{\rm uni}^{I_1}\to\B_{\rm uni}^{p^{-1}I_1}$ et 
$\varphi:\B_{\rm uni}^{(I_1)}\simeq \B_{\rm uni}^{(p^{-1}I_1)}\overset{\iota'}{\simeq} \B_{\rm uni}^{(I_2)}$.
On pose 
$$\Bpff{K}^+ :=\B^{(I_2)}_{\rm uni},\quad \Bpff{K}:=\Bpff{K}^+[\tfrac{1}{t}]$$
Alors $\iota$ fournit une identification $\iota_K:\B^{(I_1)}_{\rm uni}=\Bpff{K}^+ $,
et la restriction de $\varphi$ \`a $\Bpff{K}^+ $ 
devient un isomorphisme $\varphi_K:\Bpff{K}^+\overset{\sim}{\to}\Bpff{K}^+$
prolongeant $\varphi:\B\overset{\sim}{\to}\B$. 

\begin{rema}\label{tue55}
(i) La topologie sur $\Bpff{K}^+$ est strictement plus fine que celle induite par le plongement
dans $\B^{I_2}_{\rm uni}$. 

(ii) Il serait plus naturel de d\'efinir $\Bpff{K}^+$ comme 
$\cap_{J\supset I_2}\iota_{J}^{I_2}(\B^{J}_{\rm uni})$ mais, comme nous l'a fait
remarquer ChatGPT, ${\rm R}^1\varprojlim \iota_{J}^{I_2}(\B^{J}_{\rm uni})\neq 0$,
ce qui poserait des probl\`emes pour d\'eduire la cohomologie galoisienne de
$\Bpff{K}^+$ de celle des $\B^{J}_{\rm uni}$, si on le d\'efinissait de cette mani\`ere.
\end{rema}

\Subsubsection{Passage \`a la limite sur les extensions finies de $K$}
Il r\'esulte du lemme~\ref{tue107} que, pour tout $J\supset I_2$, l'image de
$\B^J_{\rm uni}$ dans $\B^{I_2}_{\rm uni}$ est stable par $G_F$ et pas seulement
par $G_K$; il en est de m\^eme de $\Bpff{L}^+$, pour tout~$L$.

\begin{lemm}\label{tue500}
Soient $M\subset L$ des extensions de $K$, galoisiennes, finies
et soient $F_M:=M\cap F^{\rm nr}$, $F_L:=L\cap F^{\rm nr}$. 
Soient $I,J\supset I_2$, et soit $u:\B^I_{{\rm uni},M}\to\B^J_{{\rm uni},M}$
un morphisme semi-lin\'eaire de $F_L$-alg\`ebres commutant \`a $G_K$.  Alors
$u$ admet une extension $u_L:\B^I_{{\rm uni},L}\to\B^J_{{\rm uni},L}$ commutant \`a
$G_K$.
\end{lemm}
\begin{proof}
On a $\B^I_{{\rm uni},L}=\varinjlim_n\B^I_{n,L}$, et 
$\B^I_{n,L}=\B^I_{n-1,L}[X^{(n)}_1,\cdots,X^{(n)}_{r_n}]$.
Posons $A_0:=\B^I_{{\rm uni},M}$ et
notons $A_n$ la sous-$\B^I_{{\rm uni},M}$-alg\`ebre de $\B^I_{{\rm uni},L}$ engendr\'ee par 
les $X^{(n)}_i$ (ou, ce qui revient au m\^eme, par $\B^I_{n,L}$).  Nous allons
construire, par r\'ecurrence sur $n$,
 une extension $u_n:A_n\to \B^J_{{\rm uni},L}$ commutant \`a
$G_K$.  Il n'y a rien \`a faire si $n=0$ (car $A_0=\B^I_{{\rm uni},K}$).

On suppose $u_{n-1}$ construite.  Soient $\widetilde W_n=
A_{n-1}+\sum_i F_L\cdot X^{(n)}_i$, et $W_n:=\widetilde W_n/A_{n-1}$.  
Il r\'esulte du lemme~\ref{tue107} (et de sa preuve) que
$\widetilde W_n$ est stable par $G_K$.
On choisit
une sous-famille des $X^{(n)}_i$, pour $i\in{\bf i}$, formant une base de $W_n$ sur $F_L$.
Pour construire $u_n$ il suffit de d'\'etendre $u_{n-1}$ \`a $\widetilde W_n$
en une application $F_L$-semi-lin\'eaire, $G_K$-\'equivariante, et pour cela, il
suffit de d\'eterminer les images des $X^{(n)}_i$, pour $i\in{\bf i}$ (pour passer
d'une base de $W_n$ \`a une base de transcendance de $A_n$ sur $A_{n-1}$,
on peut utiliser les lemmes~\ref{tue1} et~\ref{tue2}).

Comme $H^1(G_L,\B^J_{{\rm uni},L})=0$, on peut trouver $x^{(n)}_i\in \B^J_{{\rm uni},L}$
(unique \`a addition pr\`es d'un \'el\'ement de $(\B^J_{{\rm uni},L})^{G_L}=F_L$), 
pour $i\in{\bf i}$, tel que 
$$(\sigma-1)\cdot x^{(n)}_i=u_{n-1}((\sigma-1)\cdot X^{(n)}_i),\quad
{\text{pour tout $\sigma\in G_L$.}}$$
  Notons $u_n':\widetilde W_n\to \B^J_{{\rm uni},L}$
l'application $F_L$-semi-lin\'eaire co\"{\i}ncidant avec $u_{n-1}$ sur $A_{n-1}$
et envoyant $X^{(n)}_i$ sur $x^{(n)}_i$, pour $i\in{\bf i}$.

Faisons agir $\sigma\in G_K$ sur ${\rm Hom}_F(\widetilde W_n,\B^J_{{\rm uni},L})$
par $(\sigma* v)(x)=\sigma(v(\sigma^{-1}x))$.  Alors $(\sigma-1)*u'_n$ est $0$
si $\sigma\in G_L$, et est identiquement $0$ sur $A_{n-1}$
pour tout $\sigma\in G_K$ (puisque $u_{n-1}$ est $G_K$-\'equivariante), et est
\`a valeurs dans $F_L$. On peut donc consid\'erer $\sigma\mapsto (\sigma-1)*u'_n$
comme un $1$-cocycle sur $G_K/G_L$ \`a valeurs dans ${\rm Hom}_F(W_n,F_L)$.
Comme $G_K/G_L$ est un groupe fini, ce $1$-cocycle est un $1$-cobord,
et il existe $v_n:\widetilde W_n\to F_L$, identiquement $0$ sur $A_{n-1}$,
tel que $(\sigma-1)*u'_n=(\sigma-1)*v_n$ pour tout $\sigma\in G_K$.
Il suffit alors de poser $u_n=u'_n-v_n$ pour obtenir l'objet d\'esir\'e.
\end{proof}

Si $L$ est une extension galoisienne finie de $K$, 
on a $\Bpff{K}^+\subset\Bpff{L}^+\subset\B^{I_2}_{\rm uni}$,
et on peut imposer, gr\^ace au lemme~\ref{tue500}, que si $M\subset L$ est une extension
galoisienne de $K$,
les restrictions de $\iota_L$ et $\varphi_L$
\`a $\Bpff{M}^+$ co\"{\i}ncident avec $\iota_M$ et $\varphi_M$.
Cela permet, en \'ecrivant $\overline K$ comme la limite inductive $\varinjlim L$ 
d'une r\'eunion croissante d\'enombrable
d'extensions finies galoisiennes de $K$, de d\'efinir
$$\bpff^+=\varinjlim\nolimits_{[L:K]<\infty}\Bpff{L}^+,\quad \bpff:=\bpff^+[\tfrac{1}{t}]$$
et de les munir de $\varphi$ d\'efini comme la limite des $\varphi_L$.

\begin{rema}\label{tue101}
(i) On peut imposer que $\varphi(\log\tilde p)=p\log\tilde p$
car $\varphi((\sigma-1)\log\tilde p)=p(\sigma-1)\log\tilde p$, pour tout $\sigma\in G_K$.

(ii) L'identification $\B^{(I_1)}_{\rm uni}=\Bpff{K}^+ $
permet de construire une injection $G_K$-\'equivariante
$\Bpff{K}^+ \hookrightarrow \bpdr^+$ prolongeant
$\B \hookrightarrow \bdr^+$.  Celle-ci s'\'etend, en passant \`a la limite, 
en une injection $G_K$-\'equivariante
$\bpff^+ \hookrightarrow \bpdr^+$ et, en inversant $t$, en une injection $G_K$-\'equivariante
$\bpff \hookrightarrow \bpdr$.
\end{rema}

\Subsubsection{Cohomologie galoisienne de $\bpff$}
\begin{theo}\label{tue103}
{\rm (i)} Si $\Lambda_K=\Bpff{K}^+,\Bpff{K}$, alors
$$H^0(G_K,\Lambda_K )=F,\quad H^i(G_K,\Lambda_K )=0, {\text{si $i\geq 1$.}}$$
De plus, $N$ laisse stable $\Lambda_K \subset \B_{\rm uni}^{I_2}$.

{\rm (ii)} 
Si $\Lambda=\bpff^+,\bpff$, alors
$$H^0(G_K,\Lambda )=F,\quad H^i(G_K,\Lambda )=0, {\text{si $i\geq 1$.}}$$
De plus, $N$ laisse stable $\Lambda \subset \B_{\rm uni}^{I_2}$.
\end{theo}
\begin{proof}
On a $H^0(G_K,\B_{\rm uni}^{J,\leq m})=F$ pour tous $J$ et $m$, et donc 
$$H^0(G_K,\Bpff{K}^+ )=F,\quad
{\rm R}^1\varprojlim\nolimits_J H^0(G_K,\B_{\rm uni}^{J,\leq m})=0,\ {\text{pour tout $m$.}}$$

Maintenant, si $m$ est fix\'e, $H^1(G_K,\B_{\rm uni}^{J,\leq m})=W_J^{\otimes m+1}$
(o\`u $W_J=H^1(G_K,\B^J)$), d'apr\`es le (ii) de la rem.\,\ref{tue23},
et $H^i(G_K,\B_{\rm uni}^{J,\leq m})=0$ si $i\geq 2$.
Comme $W_J\to W_I$ est surjective, si $J\supset I$, on en d\'eduit que
${\rm R}^1\varprojlim_J H^1(G_K,\B_{\rm uni}^{J,\leq m})=0$ (le m\^eme
r\'esultat est trivialement vrai pour $H^i$, avec $i\geq 2$).
Il en r\'esulte que 
$$H^i(G_K, \B_{\rm uni}^{(I_2),\leq m})=\varprojlim\nolimits_J
H^i(G_K, \B_{\rm uni}^{J,\leq m}),\quad{\text {si $i\geq 1$.}}$$
Il s'ensuit que $H^i(G_K, \B_{\rm uni}^{(I_2),\leq m})=0$ si $i\geq 2$.
On en d\'eduit que $H^i(G_K,\Bpff{K}^+ )=0$ si $i\geq 2$, en passant \`a la limite sur $m$.

Pour prouver que $H^1(G_K,\Bpff{K}^+ )=0$, on utilise le fait que
$H^1(G_K,\B_{\rm uni}^{J,\leq m})\to H^1(G_K,\B_{\rm uni}^{J,\leq m+1})$ est identiquement~$0$,
pour tous $J$ et $m$; il en r\'esulte que $H^1(G_K, \B_{\rm uni}^{(I_2),\leq m})\to
H^1(G_K, \B_{\rm uni}^{(I_2),\leq m+1})$ est identiquement~$0$
pour tout $m$, ce qui prouve le r\'esultat en passant \`a la limite sur $m$.

Enfin, que $N$ laisse stable $\Bpff{K}^+$ est une cons\'equence directe du lemme~\ref{tue90}.

Ceci d\'emontre le (i) pour $\Bpff{K}^+$; 
le (ii) pour $\bpff^+$ s'en d\'eduit facilement par passage \`a la limite sur
les extensions finies galoisiennes $L$ de $K$.

Pour traiter le cas de $\Bpff{K}$ et $\bpff$, on utilise le lemme~\ref{gopal9} (ou plut\^ot
le (i) de la rem.\,\ref{gopal12.1}) selon lequel l'application naturelle
$H^i(G_K,\B^I)\to H^i(G_K,t^{-j}\B^I)$ 
est un isomorphisme pour tous $I$ ouvert, $i\geq 0$ et $j\geq 0$.
On en d\'eduit, par r\'ecurrence sur $n$, que 
$H^i(G_K,\B^{I,\leq n}_{\rm uni})\to H^i(G_K,t^{-j}\B^{I,\leq n}_{\rm uni})$
est un isomorphisme pour tous $I$ ouvert, $i\geq 0$, $j\geq 0$ et $n\geq 0$.
Ceci permet de ramener le calcul de $H^i(G_K,t^{-j}\Bpff{K}^+)$ \`a celui
de $H^i(G_K,\Bpff{K}^+)$; on en d\'eduit le r\'esultat pour $\Bpff{K}$ par passage \`a la
limite sur $j$, et pour $\bpff$ par passage \`a la limite sur
les extensions finies galoisiennes $L$ de~$K$.
\end{proof}

\begin{rema}\label{tue111}
(i) On a $N\, \varphi=p\,\varphi\, N$ sur le sous-anneau $\bff^+$ de
$\Bpff{K}^+$. 

(ii) Par contre, on ne peut pas avoir $N\, \varphi\neq p\,\varphi\, N$ sur $\Bpff{K}^+$.
En effet, on a $\log t\in\Bpff{K}^+$ car le $1$-cocycle associ\'e 
$\sigma\mapsto \log\chi_{\rm cycl}(\sigma)$ est \`a valeurs dans $\B_I$, pour tout $I$.
On d\'eduit de la rem.\,\ref{tue111} (pour $n=0$, puisque
$I_2\subset]\frac{1}{p},1[$) que $N(\log t)=-\frac{p}{p-1}$, 
et donc $\varphi(N(\log t))=N(\log t)$.
Par ailleurs, comme $\varphi$ commute \`a~$G_K$, on d\'eduit de la
formule $(\sigma-1)\cdot\log t=\log\chi_{\rm cycl}(\sigma)$
que $\varphi(\log t)-\log t$ est fixe par $G_K$, et donc appartient \`a $F$.
Il s'ensuit que $N(\varphi(\log t))=N(\log t)$. On en d\'eduit que 
$N(\varphi(\log t))\neq p\varphi(N(\log t))$, et donc que 
$N\, \varphi\neq p\,\varphi\, N$ sur $\Bpff{K}^+$.
\end{rema}

\begin{rema}
Soit $A=(\bpff^+)^{N=0}$.
Alors $\bpff^+=A[\log\tilde p]$, $A$ est stable par $G_K$ et $H^1(G_K,A)=F\cdot{\rm Kum}_p$,
mais le (ii) de la rem.\,\ref{tue111} implique que $A$ n'est pas stable par $\varphi$.
\end{rema}

\Subsection{La conjecture $C_{\rm st}$}\label{sect12}
\begin{conj}\label{cst2.2}
Si $Y$ est une vari\'et\'e analytique lisse, de dimension~$d$, partiellement propre,
 d\'efinie sur $K$,
on a un isomorphisme naturel de
$(\varphi, N,G_K)$-modules:
\begin{align*}
&(C_{\rm st})&
\rg_{\rm HK}(Y_{C}) &\simeq 
\varinjlim\nolimits_{[L:K]<\infty}
{\rm RHom}_{G_L}(\rg_{{\proeet},c}(Y_{C},\Q_p(d))[2d],\bpff)
\end{align*}
Cet isomorphisme est compatible avec celui de la conj.\,\ref{cst2.1}: l'injection
$\bpff\hookrightarrow\bpdr$ induit l'isomorphisme de Hoydo-Kato
$\C_p\otimes^{\Box}_{\Q_p^{\rm nr}}\rg_{\rm HK}(Y_{C})
\simeq \C_p\otimes^{\Box}_K\rg_{\rm dR}(Y)$.
\end{conj}

\begin{rema}
(i)
 Il faut faire un peu attention \`a la signification
de $\varinjlim_{[L:K]<\infty}$ : si les groupes de cohomologie de de Rham sont
de dimension finie, la formule est valable telle quelle; dans le cas g\'en\'eral, cette
condition de finitude est satisfaite localement, ce qui donne une recette pour exprimer
les $H^i_{\rm HK}(Y_{C})$.

(ii)  On a $N\varphi=p\varphi N$ sur $\rg_{\rm HK}(Y_{C})$ 
mais pas sur $\bpff$ (rem.\,\ref{tue111}). Il n'y a 
pas de contradiction car on passe de $\bff$ \`a $\bpff$ pour tuer les classes parasites;
celles qui survivent vivent dans $\bff$ sur lequel on a bien $N\varphi=p\varphi N$.
\end{rema}

\appendix

\section{Alg\`ebres de m\'elange}\label{shuffle1}
Nous expliquons comment la structure d'alg\`ebre de m\'elange permet de retrouver,
de mani\`ere plus conceptuelle, les r\'esultats des \S\S~\ref{tue221} et~\ref{sect7}, 
et de les pr\'eciser (cf.~rem.\,\ref{ass1}).

   Dans toute cette section, $F$ est un
corps non archim\'edien complet de caract\'eristique nulle, $G$ un groupe
profini et $R$ une $F$-alg\`ebre topologique commutative munie
d'une action continue de $G$ par automorphismes de $F$-alg\`ebres. 
On suppose que $R$ est localement convexe en tant que $F$-espace vectoriel.
Par ailleurs:
\begin{hypothesis}\label{hypo1}
On suppose que
$$
 R^G=F,\quad W:=H^1(G,R)\text{ est de dimension finie sur }F,
 \quad H^i(G,R)=0\ (i\ge2).
$$
\end{hypothesis}

\Subsection{L'alg\`ebre de Hopf ${\rm Sh}_F(W)$ et le groupe prounipotent ${\cal U}$} 
Nous rappelons ici la d\'efinition de l'alg\`ebre de Hopf de m\'elange $\mathrm{Sh}_F(W) $ associ\'ee \`a $W$ (voir
\cite[section 1.6]{GR}
et \cite{Reutenauer93} pour plus de d\'etails).

\vskip2mm
On munit
$\bigoplus_{n\ge0}W^{\otimes n}$
muni de la topologie de somme directe localement convexe.
Si $w$ est un tenseur pur de longueur $n\geq 1$ (i.e.~$w=w_1\otimes\cdots \otimes w_n$), on note $[w]$ ou
$[w_1|\cdots|w_n]$ l'\'el\'ement correspondant de $\bigoplus_{n\ge0}W^{\otimes n}$ et $\ell(w)=n$
sa longueur; on identifie
$W^{\otimes 0}$ \`a $F$ et un tenseur vide (i.e.~de
longueur~$0$) \`a $1\in F$.

\vskip2mm
Il y a deux produits usuels 
qui transforment $\bigoplus_{n\ge0}W^{\otimes n}$ en $F$-alg\`ebre unitaire:

$\bullet$ Le {\it produit de concat\'enation} pour lequel 
$$[w_1|\cdots|w_n]\cdot [w'_1|\cdots|w'_m]=
[w_1|\cdots|w_n| w'_1|\cdots|w'_m]$$
L'alg\`ebre ainsi obtenue est {\it l'alg\`ebre tensorielle $T_F(W)$ associ\'ee \`a $W$}.

$\bullet$ Le {\it produit de m\'elange} $\shprod$ d\'efini par 
$$[v_1|\cdots|v_n]\shprod [v'_1|\cdots|v'_m]=\sum_Iw_I$$ 
o\`u $I$ parcourt les parties \`a $n$ \'el\'ements de $\{1,\dots,n+m\}$ et $w_I$ est obtenu de la mani\`ere suivante:
on num\'erote $i(1)<\dots<i(n)$ les \'el\'ements de $I$ et $j(1)<\dots<j(m)$ ceux de son compl\'ementaire,
et on a $w_I=[w_1|\cdots|w_{n+m}]$, avec $w_k=v_a$ si $k=i(a)\in I$ et $w_k=v'_b$ si $k=j(b)\notin I$.
L'alg\`ebre ainsi obtenue est {\it l'alg\`ebre de m\'elange ${\rm Sh}_F(W)$
associ\'ee \`a $W$}.

Les deux produits sont compatibles avec la filtration croissante associ\'ee \`a la longueur:
i.e.~${\rm Fil}_i=\bigoplus_{n\leq i}W^{\otimes n}$ et
${\rm Fil}_i*{\rm Fil}_j\subset{\rm Fil}_{i+j}$, si $*=\cdot,\shprod$.

\vskip2mm
On transforme ${\rm Sh}_F(W)$ en une alg\`ebre de Hopf commutative pour le produit $\shprod$, le coproduit $\Delta$,
la co-unit\'e $\epsilon$ et l'antipode $[w_1|\cdots|w_n]\mapsto (-1)^n[w_n|\cdots|w_1]$, o\`u:

$\bullet$
$\Delta:{\rm Sh}_F(W)\to {\rm Sh}_F(W)\otimes{\rm Sh}_F(W)$ est donn\'e 
par la d\'econcat\'enation\footnote{La notation $\sum_{w=vv'}$ est un peu abusive; de
mani\`ere pr\'ecise, si $w=w_1\otimes\cdots\otimes w_n$, le premier terme
est $[w]\otimes 1$, le dernier $1\otimes[w]$ et les termes interm\'ediaires
$[w_1|\cdots|w_i]\otimes[w_{i+1}|\cdots|w_n]$, pour $1\leq i\leq n-1$.}:
$$
 \Delta([w])=\sum\nolimits_{w=vv'}[v]\otimes[v'].
$$

$\bullet$ $\epsilon(1)=1$ et $\epsilon([v])=0$ si $\ell([v])\geq 1$.

\vskip1mm
On d\'efinit l'alg\`ebre de Lie libre
$\mathbb L_F(W^\vee)$ comme la sous-alg\`ebre de Lie de 
  $T_F(W^\vee)$ engendr\'ee par $W^\vee$,
le crochet sur $T_F(W^\vee)$ \'etant le commutateur
$[a,b]:=a\cdot b-b\cdot a$
pour le produit de concat\'enation.
Notons
  $\mathfrak l_n:=\mathbb L_F(W^\vee)\cap (W^\vee)^{\otimes n}$
sa composante homog\`ene de degr\'e $n$ ; elle est engendr\'ee
par les crochets it\'er\'es faisant intervenir $n$ \'el\'ements
de $W^\vee$.

\vskip1mm
Soit ${\cal U}$ le groupe pro-unipotent
$\mathcal U:=\operatorname{Spec}\operatorname{Sh}_F(W)$.
Si $R$ est une $F$-alg\`ebre, on note 
$$\langle\ ,\ \rangle:{\cal U}(R)\times{\rm Sh}_F(W)\to R$$
l'accouplement naturel.
${\cal U}$
a pour alg\`ebre de Lie
\[
  \mathfrak u
  :=\operatorname{Lie}(\mathcal U)
  =\widehat{\mathbb L}_F(W^\vee)
  =\prod\nolimits_{n\geq1}\mathfrak l_n,
\]
o\`u la compl\'etion est prise par rapport au degr\'e.
Pour $m\geq1$, on note $\mathcal U_{\le m}$ le quotient
unipotent de $\mathcal U$ dont l'alg\`ebre de Lie est
\[
{\rm Lie}(\mathcal U_{\le m})=  \mathfrak u_{\le m}
  :=\mathfrak u\big/\prod\nolimits_{n>m}\mathfrak l_n
  \simeq\bigoplus_{n=1}^{m}\mathfrak l_n
\]
On a une suite exacte $1\to {\goth l}_m\to \mathcal U_{\le m}\to \mathcal U_{\le m-1}\to 1$
de groupes unipotents, et ${\goth l}_m$ est incluse dans le centre de $\mathcal U_{\le m}$.

\begin{rema}\label{NOGI2}
Si $v\in{\cal U}(F)$, l'application $u\mapsto vu$ dans ${\cal U}$ induit un morphisme
de $F$-alg\`ebres $h_v:{\rm Sh}_F(W)\to{\rm Sh}_F(W)$ tel que $vu=u\circ h_v$.  Or
$$\langle vu,[w]\rangle=\big\langle (v,u),\Delta([w])\big\rangle=
\Big\langle u,\sum\nolimits_{w=w'w''}\langle v,[w']\rangle [w'']\Big\rangle$$
Il s'ensuit que $h_v$ est
donn\'e par 
\begin{equation}
h_v([w])=(v\otimes 1)\Delta([w])=\sum\nolimits_{w=w'w''}\langle v,[w']\rangle [w'']
\end{equation}
La formule montre que $h_v({\rm Fil}_n)\subset {\rm Fil}_n$ et
que $h_v$ induit l'identit\'e sur ${\rm Fil}_n/{\rm Fil}_{n-1}$
(en effet, $\ell(w'')<\ell(w)$ si $w=w'w''$ et $w''\neq w$);
il en r\'esulte, par r\'ecurrence, que $h_v({\rm Fil}_n) = {\rm Fil}_n$.
\end{rema}

\begin{rema}D'apr\`es \cite[th.~3.1.1(e)]{Radford79} ou
\cite[section 4]{MR89},
l'alg\`ebre de m\'elange $\operatorname{Sh}_F(W)$ est une alg\`ebre de polyn\^omes.
 On peut munir les variables d'un poids
de sorte que la filtration par la longueur co\"{\i}ncide avec
la filtration par le degr\'e total pond\'er\'e.
\end{rema}
\Subsection{Structures galoisiennes}
D\'efinissons la $R$-alg\`ebre $P$ et ses sous-$R$-modules
libres de type fini $P_{\leq n}$ par:
\[
  P:=R\otimes_F\operatorname{Sh}_F(W)=\varinjlim\nolimits_n P_{\leq n},\quad
  P_{\le n}:=
  R\otimes_F\Big(\bigoplus_{j=0}^{n}W^{\otimes j}\Big).
\]
Les $P_{\leq n}$ d\'efinissent une filtration croissante exhaustive de $P$,
compatible au produit de m\'elange au sens o\`u
\[
  P_{\le m}\shprod P_{\le n}\subseteq P_{\le m+n}.
\]

On munit $P$ d'une structure de $G$-module en faisant agir $\sigma\in G$ sur $R$;
cette action laisse stables les $P_{\leq n}$, et on a $\sigma([w])=[w]$ pour tout $w$.
On en d\'eduit une action de $G$, dite {\og naturelle\fg},
sur ${\cal U}(R)$: 
on demande que $\langle \sigma(x),\sigma(\lambda)\rangle=\sigma(\langle x ,\lambda \rangle)$
pour tous $x\in{\cal U}(R)$, $\lambda\in P$ et $\sigma\in G$. 

Si $\sigma\mapsto C_\sigma$ est un $1$-cocycle sur $G$ (i.e.~$C_{\sigma\tau}=C_\sigma\sigma(C_\tau)$), 
\`a valeurs dans ${\cal U}(R)$, on peut tordre l'action naturelle de $G$ sur ${\cal U}(R)$, et obtenir
une nouvelle action via la formule $\sigma*_C x=\sigma(x)C_\sigma^{-1}$.  Ceci d\'efinit
une action tordue sur $P$ (on note $P_C$ la $R$-alg\`ebre $P$ munie de cette action tordue $*_C$), 
et on a (en utilisant $\sigma(C_{{\sigma}^{-1}})=C_\sigma^{-1}$)
\begin{align*}
\sigma^{-1}(\langle x,\sigma*_C[w]\rangle)&=\langle\sigma^{-1}*_C x,[w]\rangle=
\langle\sigma^{-1}(x)C_{\sigma^{-1}}^{-1},[w]\rangle=
\langle(\sigma^{-1}(x),C_{\sigma^{-1}}^{-1}),\Delta([w])\rangle\\
&=\sum_{w=vv'}\langle\sigma^{-1}(x),[v]\rangle \langle C_{\sigma^{-1}}^{-1},[v']\rangle=
\sigma^{-1}\Big(\sum_{w=vv'}\langle x,[v]\rangle\,\langle C_\sigma,[v']\rangle\Big)
\end{align*}
On en d\'eduit que, si $w=w_1\otimes\cdots\otimes w_n$,
\begin{align}\label{duck11}
\sigma*_C[w]&=\sum\nolimits_{w=vv'}\langle C_\sigma,[v']\rangle\,[v]\notag\\
&=[w]+\langle C_\sigma,[w_n]\rangle\, [w_1|\cdots|w_{n-1}]+\cdots+ \langle C_\sigma,[w]\rangle
\end{align}
\begin{rema}\label{duck12}
On d\'eduit de cette formule que $\sigma*_C[w]-[w]\in P_{\leq n-1}$ si $w$ est de longueur $n$
(le seul terme de longueur~$> n-1$ dans la formule ci-dessus est le premier).
En particulier $G$ laisse stables les $P_{\leq n}$ (on note $P_{C,\leq n}$
le $G$-module ainsi obtenu), et l'action de $G$ sur
$P_{C,\leq n}/P_{C,\leq n-1}
=R\otimes_F W^{\otimes n}$ co\"{\i}ncide avec l'action naturelle.
\end{rema}
\Subsection{Construction d'un cocycle}
Choisissons un scindage $c:W\to Z^1(G,R)$, $F$-lin\'eaire,
 de la projection $Z^1(G,R)\to H^1(G,R)=W$.
Il
d\'etermine un $1$-cocycle
$$ \tilde{c}:\quad G\to {\mathrm{Hom}}_F(W,R) \simeq R\otimes_F W^\vee={\cal U}_{\leq 1}(R), 
\qquad \tilde{c}_g(w)=c(w)(g). $$

\begin{theo}\label{duck1}
Supposons l'hypoth\`ese~\ref{hypo1} satisfaite. Alors: 

{\rm (i)} Il existe un $1$-cocycle
$ C:G\longrightarrow\mathcal U(R)$
dont l'image dans
$ \mathcal U_{\le1}(R) $ est $\tilde{c}$.

{\rm (ii)}
Pour tout $n\ge0$, on a\footnote{Rappelons que l'on note $P_C$, $P_{C,\leq n}$
les $G$-modules obtenus en tordant l'action naturelle par~$C$.} :
$$
 H^0(G,P_{C,\le n})=F,\quad
 H^1(G,P_{C,\le n})\simeq W\otimes_F W^{\otimes n},\quad
 H^i(G,P_{C,\le n})=0\ (i\ge2).
$$
De plus, $H^1(G,P_{C,\le n-1})\to H^1(G,P_{C,\le n})$ est identiquement $0$, et 
$\mathrm{R}\Gamma(G,P_C)\simeq F[0]$.
\end{theo}
\begin{proof}
Pour prouver le (i), il s'agit de construire une suite compatible $(c_m)_{m\geq 1}$
de $1$-cocyles, o\`u $c_m$ est \`a valeurs dans ${\cal U}_{\leq m}(R)$,
et $c_1=\tilde c$.
Supposons $c_{m-1}$ construit et relevons $c_{m-1}$ en une $1$-cocha\^{\i}ne continue $\hat c_m$
\`a valeurs dans ${\cal U}_{\leq m}(R)$.
Comme l'extension $0\to R\otimes{\goth l}_m\to {\cal U}_{\leq m}(R)\to {\cal U}_{\leq m-1}(R)\to 1$
est centrale, 
$(\sigma,\tau)\mapsto\sigma(\hat c_{m,\tau})\hat c_{m,\sigma\tau}^{-1}\hat c_{m,\sigma}$
est un $2$-cocycle \`a valeurs dans $R\otimes_F\mathfrak l_m$.
Or on a
\[
 H^2(G,R\otimes_F  {\goth l}_m)\simeq
 H^2(G,R)\otimes_F {\goth l}_m=0,
\]
On peut donc multiplier $\hat c_m$ par une $1$-cocha\^{\i}ne \`a valeurs dans 
$R\otimes_F  {\goth l}_m$ pour la transformer en $1$-cocycle. Ceci d\'emontre le (i).

Passons \`a la preuve du (ii).
On dispose de suites exactes de $G$-modules
\begin{equation}\label{opera1}
 0\longrightarrow P_{C,\le n-1}\longrightarrow P_{C,\le n}
 \longrightarrow R\otimes_FW^{\otimes n}\longrightarrow0.
\end{equation}
On en d\'eduit par r\'ecurrence, en utilisant l'hypoth\`ese \ref{hypo1}, que
$H^i(G,P_{C,\le n})=0$ si $i\geq 2$. La suite exacte longue de cohomologie
associ\'ee \`a la suite~(\ref{opera1}) r\'eduit la preuve des
\'enonc\'es concernant $H^0$ et $H^1$
au fait que
l'application de connexion $W^{\otimes n}=H^0(G,R\otimes_FW^{\otimes n})\to
H^1(G,P_{C,\le n-1})$ est un isomorphisme.
Or l'hypoth\`ese de r\'ecurrence fournit un isomorphisme
$H^1(G,P_{C,\le n-1})\simeq H^1(G,R)\otimes W^{\otimes n-1}$ (donn\'e par
la projection $P_{\leq n-1}\to R\otimes W^{\otimes n-1}$), et on d\'eduit de
l'identification des termes de longueur $n-1$ dans
la formule (\ref{duck11}) que l'application de connection
est $[w_1|\cdots|w_n]\mapsto c(w_n)\otimes[w_1|\cdots|w_{n-1}]$.
Comme $c:W\to H^1(G,R)$ est un isomorphisme, cela implique que
l'application de connexion est effectivement un isomorphisme.

Comme $\mathrm{R}\Gamma(G,P_C)\simeq \varinjlim_n\mathrm{R}\Gamma(G,P_{C,\le n})$
(car $G$ est compact et $P_C$ est, topologiquement, la limite inductive des
$P_{C,\leq n}$), cela ach\`eve la preuve du th\'eor\`eme.
\end{proof}

\Subsection{R\'ealisations dans une alg\`ebre acyclique}
Soient $T$ une $F$-alg\`ebre munie d'une action de $G$, et $\iota:R\to T$ un
morphisme $G$-\'equivariant de $F$-alg\`ebres.
Un morphisme $u:P_C\to T$ de $R$-alg\`ebres n'est autre qu'un point de ${\cal U}_C(T)$
et $u$ est $G$-\'equivariant si et seulement si il est fixe par $G$ vu comme
\'el\'ement de ${\cal U}_C(T)$, ce qui se traduit par $\sigma(u)\iota(C_\sigma)^{-1}=u$
pour tout $\sigma\in G$.
\begin{prop}\label{NOGI1}
{\rm (i)} Si $H^1(G,T)=0$, alors $\iota$ admet un prolongement $G$-\'equivariant $u:P_C\to T$.

{\rm (ii)} Si $u_1,u_2$ sont deux tels prolongements et si $H^0(G,T)=F$,
 il existe $v\in{\cal U}(F)$,
unique, tel que $u_2=vu_1\in {\cal U}_C(T)$.
\end{prop}
\begin{proof}
Pour prouver le (i),
il suffit de prouver que ${\cal U}_C(T)$ contient un \'el\'ement $G$-invariant et, pour cela,
il suffit de construire une tour $(u_m)_{m\geq 0}$ d'\'el\'ements
compatibles avec $u_m\in {\cal U}_{C,\leq m}(T)^G$. 
Il n'y a rien \`a faire pour $m=0$.

Supposons $u_{m-1}$ construit.  
La fibre de ${\cal U}_{C,\leq m}(T)\to {\cal U}_{C,\leq m-1}(T)$ 
au-dessus de $u_{m-1}$ est non vide et,
si $u$ est un rel\`evement fix\'e, tout autre rel\`evement
s'\'ecrit de mani\`ere unique $\exp(a)u$, avec
$a\in T\otimes_F\mathfrak l_m$.
Comme $u_{m-1}$ est fixe par $G$, cette fibre est stable
par l'action $*_C$, et on a
$ g*_C(\exp(a)u)=\exp(g(a))(g*_C u) $,
o\`u $G$ agit sur $T\otimes_F\mathfrak l_m$ par son action
sur $T$.
L'obstruction \`a relever $u_{m-1}$
en un \'el\'ement fixe par $G$ vit donc dans $H^1(T\otimes{\goth l}_m)$. Or ce groupe est $0$
puisque $H^1(G,T)=0$; on peut donc trouver $u_m\in {\cal U}_{C,\leq m}(T)^G$ relevant $u_{m-1}$.
On en d\'eduit le (i). 

Pour le (ii), la condition selon laquelle $u_1$ et $u_2$ sont
fixes par $G$ (vus comme \'el\'ements de ${\cal U}_C(T)$) se traduit par
$u_i=\sigma(u_i)\iota(C_\sigma)^{-1}$
pour tout $\sigma\in G$, si $i=1,2$. Il en r\'esulte que $u_2u_1^{-1}=
\sigma(u_2u_1^{-1})$ pour tout $\sigma\in G$, et donc $v:=u_2u_1^{-1}\in {\cal U}(F)$
(gr\^ace \`a l'hypoth\`ese $H^0(G,T)=F$),
ce qui permet de conclure puisque $u_2=vu_1$.
\end{proof}

  Soient $R,R'$ des $F$-alg\`ebres satisfaisant l'hypoth\`ese \ref{hypo1} 
(pour le m\^eme $G$), et soient $P:=P_C$, $P':=P'_C$ correspondant \`a $R,R'$
comme dans le th.\,\ref{duck1}. 
D'apr\`es la proposition \ref{NOGI1}, tout morphisme continu $G$-\'equivariant 
de $F$-alg\`ebres
$ \iota:R\longrightarrow R'$
se prolonge en un morphisme $G$-\'equivariant d'alg\`ebres $P\to P'$. 
\begin{prop}\label{nogi1a}
{\rm (i)}
Tout tel prolongement $\hat\iota$ v\'erifie $\hat\iota(P_{\leq n})\subset P'_{\leq n}$, si $n\geq 0$.

{\rm (ii)} Si $\iota$ est un isomorphisme, tout tel
prolongement $\hat\iota$ est un isomorphisme, 
et $\hat\iota(P_{\leq n})= P'_{\leq n}$, si $n\geq 0$.
\end{prop}
\begin{proof}
Observons d'abord que
\begin{equation}\label{mist12}
y\in P',\quad
\bigl(\sigma \mapsto(\sigma -1)y\bigr)\in Z^1(G,P'_{\le r})
\quad\Longrightarrow\quad y\in P'_{\le r+1}.
\end{equation}
En effet, le th\'eor\`eme~\ref{duck1} affirme que l'application
$H^1(G,P'_{\le r})\to H^1(G,P'_{\le r+1})$ est nulle.
Il existe donc $z\in P'_{\le r+1}$ tel que
$(\sigma -1)z=(\sigma -1)y$ pour tout $\sigma $.
Alors $y-z\in (P')^G=F\subseteq P'_{\le0}$, ce qui prouve
\eqref{mist12}.

Soit $\hat\iota:P\to P'$ un prolongement \'equivariant.
On raisonne par r\'ecurrence sur la longueur. En longueur~$0$, on a
$\hat\iota=\iota:R\to R'$.
Si l'assertion est vraie pour les longueurs~$<d$, alors, pour $[w]$
de longueur $d$, 
$\sigma\mapsto
(\sigma -1)\hat\iota([w])=\hat\iota\bigl((\sigma -1)[w]\bigr)$
qui est est un $1$-cocycle, est \`a valeurs dans $P'_{\le d-1}$ d'apr\`es la formule
 \eqref{duck11} et l'hypoth\`ese de r\'ecurrence.
L'implication~\eqref{mist12} donne
$\hat\iota([w])\in P'_{\le d}$.
 Ceci prouve le (i).  

Si $\iota$ est un isomorphisme, il existe $\hat\iota':P'\to P$ prolongeant $\iota^{-1}$.
Mais alors $\hat\iota'\circ\hat\iota$ est un morphisme $G$-\'equivariant de $P$ dans lui-m\^eme.
Le (ii) de la prop~\ref{NOGI1} appliqu\'e \`a $u_1={\rm id}$ et $u_2=\hat\iota'\circ\hat\iota$,
et la rem.\,\ref{NOGI2} impliquent que $\hat\iota'\circ\hat\iota$ est un isomorphisme
et $\hat\iota'\circ\hat\iota(P_{\leq n})= P_{\leq n}$.  Pour la m\^eme raison,
$\hat\iota\circ\hat\iota'$ est un isomorphisme
et $\hat\iota\circ\hat\iota'(P'_{\leq n})= P'_{\leq n}$.
On en d\'eduit que $\hat\iota'$ et $\hat\iota$ sont des isomorphismes,
et que les inclusions $\hat\iota'(P'_{\leq n})\subset P_{\leq n}$
et $\hat\iota(P_{\leq n})\subset P'_{\leq n}$ sont des \'egalit\'es.
Ceci prouve le (ii) et permet de conclure.
\end{proof}

\Subsection{Comparaisons} 
Soit $A$ une $R$-alg\`ebre munie d'une action de $G$ compatible
avec celle de $G$ sur $R$, et admettant une pr\'esentation
\[
 A=\varinjlim\nolimits_{n\geq0}A_n,\qquad A_0=R,\qquad
 A_{n+1}=A_n[X_{n,\alpha}\mid\alpha\in S_n],
\]
o\`u chaque $A_n$ est stable par $G$ et
$
 \sigma(X_{n,\alpha})=X_{n,\alpha}+z_{n,\alpha}(\sigma),$ pour $
 z_{n,\alpha}\in Z^1(G,A_n).
$
Supposons que $A^G=F$, $H^1(G,A)=0$. D'apr\`es la prop.\,\ref{NOGI1}, l'inclusion
$R\hookrightarrow A$ s'\'etend (de mani\`ere non unique) en un morphisme
$G$-\'equivariant $i:P_C\to A$.

\begin{prop}
\label{comp1}
Tout morphisme continu $G$-\'equivariant
de $R$-alg\`ebres entre $P_C$ et $A$, dans l'un ou l'autre sens, est un
isomorphisme d'inverse continu.
\end{prop}
\begin{proof}
On dispose d\'ej\`a de $i:P_C\to A$; 
construisons un morphisme continu dans l'autre sens $r:A\to P_C$.
Nous montrerons ensuite que $r$ et $i$ sont des isomorphismes continus.

Pour construire $r$, on va construire sa restriction $r_n:A_n\to P_C$ par r\'ecurrence sur $n$.
On n'a pas le choix pour $r_0:R\to P_C$.
Supposons 
$r_n:A_n\to P$ construit. Pour tout $\alpha\in S_n$, choisissons
$x_{n,\alpha}\in P$ tel que
\[
 (\sigma-1)x_{n,\alpha}=r_n\bigl(z_{n,\alpha}(\sigma)\bigr),
 \qquad \sigma\in G.
\]
C'est possible car, d'apr\`es le th\'eor\`eme~\ref{duck1}, $H^1(G,P)=0$.
Envoyer $X_{n,\alpha}$ sur $x_{n,\alpha}$ prolonge $r_n$ en un morphisme $G$-\'equivariant
$r_{n+1}:A_{n+1}\to P$, continu.  On obtient $r$ en passant \`a la limite.

Alors $u:=r\circ i:P_C\to P_C$ est un endomorphisme $G$-\'equivariant de $R$-alg\`ebres.
D'apr\`es le corollaire~\ref{nogi1a}, c'est un
automorphisme d'inverse continu.
Quitte \`a remplacer $i$ par $i\circ u^{-1}$, on peut supposer que $
 r\circ i=\mathrm{id}_P. $
Alors $e:=i\circ r$ est un endomorphisme continu \'equivariant de $R$-alg\`ebres
de $A$ v\'erifiant $e^2=e$. Pour conclure, il s'agit de prouver que $e={\rm id}$.

Montrons par r\'ecurrence que $e$ fixe tout \'el\'ement de chaque $A_n$. C'est vrai
pour $A_0=R$. Supposons-le pour $A_n$; soit $\alpha\in S_n$.
Comme $z_{n,\alpha}(\sigma)\in A_n$, on a
\[
 (\sigma-1)\bigl(e(X_{n,\alpha})-X_{n,\alpha}\bigr)
 =e\bigl(z_{n,\alpha}(\sigma)\bigr)-z_{n,\alpha}(\sigma)=0,
 \qquad \sigma\in G.
\]
Donc $c:=e(X_{n,\alpha})-X_{n,\alpha}\in A^G=F$. 
Comme $e$ fixe $F$, et comme $e^2=e$, on a
 $c=e(c)=e^2(X_{n,\alpha})-e(X_{n,\alpha})=0$, et donc
$e(X_{n,\alpha})=X_{n,\alpha}$, ce qui permet de conclure.
\end{proof}

La proposition ci-dessus permet de faire le lien entre l'alg\`ebre de m\'elange $P_C$
et l'alg\`ebre $R_{\rm uni}$ du th.\,\ref{tue21} (dont l'existence ne demande que 
la condition $H^0(G,R)=F$, cf.~rem.\,\ref{tue22}).
\begin{coro}\label{kicius1}
Il existe un isomorphisme $G$-\'equivariant de $R$-alg\`ebres $P_C\simeq R_{\rm uni}$, et tout
tel isomorphisme 
identifie la filtration par la longueur sur $P_C$
 \`a la filtration par le poids sur $R_{\rm uni}$.
\end{coro}
\begin{proof} 
L'existence r\'esulte du th.\,\ref{tue21} et de la proposition~\ref{comp1}.  
Maintenant, soit $f:P_C\overset{\sim}{\to} R_{\rm uni}$
un $R$-isomorphisme topologique 
$G$-\'equivariant. On veut prouver que $f$ et $f^{-1}$ pr\'eservent les filtrations.
Rappelons que, si
$T=P_C, R_{\mathrm{uni}}$ est
muni de la filtration idoine, on a (cf.~(\ref{mist12}) pour $P_C$ et (ii) du lemme~\ref{tue0}
pour $R_{\rm uni}$)
\begin{equation}\label{opera3}
y\in T,\quad
\bigl(\sigma\mapsto(\sigma-1)y\bigr)\in Z^1(G,T_{\le n-1})
\quad\Longrightarrow\quad y\in T_{\le n}.
\end{equation}
Dans les deux cas, on va utiliser cette implication pour prouver, par r\'ecurrence sur~$n$,
que le cran $n$ de la filtration est pr\'eserv\'e.  On a $P_{C,\leq 0}=R=R_{\rm uni}^{\leq 0}$;
le r\'esultat est donc imm\'ediat pour $n=0$.  Supposons qu'il soit vrai pour $n-1$.

$\bullet$  Pour $f$, si $[w]$ est de longueur $n$, alors $(\sigma-1)[w]\in P_{C,\le n-1}$
d'apr\`es la rem.\,\ref{duck12}.
L'hypoth\`ese de r\'ecurrence donne $(\sigma-1)f([w])
=f\bigl((\sigma-1)[w]\bigr)\in R_{\mathrm{uni}}^{\le n-1}$,
et le crit\`ere (\ref{opera3}) implique $f([w])\in R_{\mathrm{uni}}^{\le n}$.
L'hypoth\`ese de r\'ecurrence est donc vraie pour $n$.

$\bullet$ Pour $f^{-1}$, si $1\leq i\leq r_n$, alors
$(\sigma-1)X_i^{(n)}=c_\sigma^{(n,i)} \in R_{\mathrm{uni}}^{\le n-1}$.
L'hypoth\`ese de r\'ecurrence donne
$(\sigma-1)f^{-1}(X_i^{(n)}) =f^{-1}(c_\sigma^{(n,i)})\in P_{C,\le n-1}$,
et le crit\`ere (\ref{opera3}) implique $f^{-1}(X_i^{(n)}) \in P_{C,\le n}$.
L'hypoth\`ese de r\'ecurrence est donc vraie pour $n$.

Ceci permet de conclure.
\end{proof}

\begin{rema}\label{beurk}
La prop.\,\ref{comp1} permet de montrer qu'on peut aussi construire
$R_{\rm uni}$ de mani\`ere plus brutale, comme la limite inductive
d'alg\`ebres $A_n$ telles que $H^1(G,A_n)\to H^1(G,A_{n+1})$ soit identiquement $0$,
pour tout $n\geq 0$ (la diff\'erence avec la construction du \S\,\ref{tue221} est
que $A_{n+1}$ est, en g\'en\'eral, de type infini sur $A_n$, si $n\geq 1$).
On pose $A_0=R$ et, $A_{n}$ \'etant construite, on choisit
des $1$-cocycles
$(z_{n,\alpha})_{\alpha\in S_n}$ dont les classes forment une
$F$-base de $H^1(G,A_n)$, et on pose
$$
A_{n+1}=A_n[X_{n,\alpha}\mid\alpha\in S_n],\qquad
\sigma(X_{n,\alpha})=X_{n,\alpha}+z_{n,\alpha}(\sigma),\qquad
A_\infty=\varinjlim\nolimits_n A_n.
$$
On a $H^1(G,A_\infty)=0$ par construction et $A_\infty^G=F$ par une variante (facile)
des arguments de la preuve du lemme~\ref{tue0}.  On peut donc utiliser
la prop.\,\ref{comp1} pour en d\'eduire que $P_C\simeq A_\infty$, et le cor.\,\ref{kicius1}
pour en d\'eduire que $A_\infty\simeq R_{\rm uni}$.
\end{rema}

\end{document}